\documentclass[12pt]{amsart}
\usepackage[active]{srcltx}
\usepackage{calc,amssymb,amsthm,amsmath,amscd, eucal,ulem}
\usepackage{alltt}
\usepackage[left=1.35in,top=1.25in,right=1.35in,bottom=1.25in]{geometry}
\RequirePackage[dvipsnames,usenames]{color}

\input{kmacros3.sty}
\input{xy}
\xyoption{all}
\usepackage{tikz}

\numberwithin{equation}{theorem}

\renewcommand{\m}{\mathfrak{m}}
\renewcommand{\n}{\mathfrak{n}}

\DeclareMathOperator{\edim}{edim}

\DeclareMathOperator{\depth}{depth}
\DeclareMathOperator{\pd}{pd}

\DeclareMathOperator{\Ass}{Ass}

\DeclareMathOperator{\chern}{ch}
\DeclareMathOperator{\gr}{gr}

\usepackage{fullpage}

\usepackage{setspace}
\usepackage{hyperref}
\usepackage{enumerate}
\usepackage{graphicx}

\usepackage[all,cmtip]{xy}
\usepackage{verbatim}

\theoremstyle{theorem}

\begin{document}
\title{Lech's conjecture in dimension three}
\author{Linquan Ma}
\address{Department of Mathematics\\ University of Utah\\ Salt Lake City\\ Utah 84112}
\email{lquanma@math.utah.edu}
\thanks{\hspace{1em} The author is partially supported by NSF Grant DMS \#1600198, and NSF CAREER Grant DMS \#1252860/1501102.}
\maketitle

\begin{center}
{\textit{Dedicated to Professor Craig Huneke on the occasion of his 65th birthday}}
\end{center}

\begin{abstract}
Let $(R,\m)\to (S,\n)$ be a flat local extension of local rings. Lech conjectured around 1960 that there should be a general inequality $e(R)\leq e(S)$ on the Hilbert-Samuel multiplicities \cite{Lechnoteonmultiplicity}. This conjecture is known when the base ring $R$ has dimension less than or equal to two \cite{Lechnoteonmultiplicity}, and remains open in higher dimensions. In this paper, we prove Lech's conjecture in dimension three when $R$ has equal characteristic. In higher dimension, our method yields substantial partial estimate: $e(R)\leq (d!/2^d)\cdot e(S)$ where $d=\dim R\geq 4$, in equal characteristic.
\end{abstract}

%\vspace{0.5em}

%\noindent R\'{e}sum\'{e}: La conjecture de Lech en dimension trois. Soit $(R,\m)\to (S,\n)$ une homomorphisme locale plat des anneaux locaux. Lech conjectur\'{e} en 1960 que $e(R)\leq e(S)$ sur les multiplicit\'{e}s de Hilbert-Samuel \cite{Lechnoteonmultiplicity}. Cette conjecture est connue lorsque l'anneau de base $R$ est de dimension inf\'{e}rieure ou \'{e}gale \`{a} deux \cite{Lechnoteonmultiplicity}, et reste ouverte dans les dimensions sup\'{e}rieures. Dans cet article, nous prouvons la conjecture de Lech en dimension trois lorsque $R$ est \'{e}quicaract\'{e}ristique. En dimension sup\'{e}rieure, nous prouvons estimation partielle substantielle: $e(R)\leq (d!/2^d)\cdot e(S)$ pour $d=\dim R\geq 4$, en \'{e}quicaract\'{e}ristique.

\setcounter{tocdepth}{1}
\tableofcontents

\section{Introduction}

Around 1960, Lech made the following remarkable conjecture on the Hilbert-Samuel multiplicities \cite{Lechnoteonmultiplicity}:
\begin{conjecture}[Lech's conjecture]
\label{Lech's Conjecture} Let $(R,\m)\rightarrow (S,\n)$ be a flat local extension of local rings. Then $e(R)\leq e(S)$.
\end{conjecture}
We note that the Hilbert-Samuel multiplicity is a classical invariant that measures the singularity of $R$. Morally speaking, the larger the multiplicity, the worse the singularity. It is very natural to expect that if $(R,\m)\to (S,\n)$ is a flat local extension, then $R$ cannot have a worse singularity than $S$. Hence, Lech's conjecture seems quite natural and interesting. However, the conjecture has now stood for over fifty years and remains open in most cases, with the best partial results still those proved in Lech's original two papers \cite{Lechnoteonmultiplicity},\cite{Lechinequalitiesofflatcouples}. There the conjecture was proved in the following cases:
\begin{enumerate}
\item $\dim R \leq 2$ \cite{Lechnoteonmultiplicity};
\item $S/\m S$ is a complete intersection \cite{Lechnoteonmultiplicity}, \cite{Lechinequalitiesofflatcouples}.
\end{enumerate}

Lech's conjecture has caught great interests to commutative algebraists, and some partial positive answers have been obtained. For example,
it follows from results of \cite{HerzogUlrichBacklinLinearMCMoverStrictcompleteintersections} that Lech's conjecture holds when the base ring $R$ is a {\it strict complete intersection}: the associated graded ring $\gr_\m R$ is a complete intersection. The conjecture was also proved when $R$ is a three-dimensional $\mathbb{N}$-graded $K$-algebra generated over $K$ by one forms for $K$ a perfect field of characteristic $p>0$ \cite{HanesThesis}. Moreover, partial results were obtained when we put various conditions on the closed fibre $S/\m S$ \cite{HerzogBLechHironakainequalities}, \cite{HerzogBKodairaSpencermap}. We refer to \cite{HanesThesis}, \cite{HanesLengthapproximationsforIndependentlygeneratedideals},  \cite{HerzogBKodairaSpencermap} and \cite{MaThesis} for other related results on Lech's conjecture.

However, despite these partial results, to the best of our knowledge Lech's conjecture remains open as long as $\dim R\geq 3$. Our main theorem in this paper settles Lech's conjecture in dimension three in equal characteristic.
\begin{theorem}
\label{theorem--main theorem}
Let $(R,\m)\rightarrow (S,\n)$ be a flat local extension between local rings of equal characteristic. If $\dim R=3$, then $e(R)\leq e(S)$.
\end{theorem}

In \cite{Lechnoteonmultiplicity}, Lech proved that, in general, we have $e(R)\leq d!\cdot e(S)$ for $(R,\m)\to (S,\n)$ flat local extension with $d=\dim R$. Our main technical result {\it greatly} generalizes this inequality in equal characteristic, and from which Theorem \ref{theorem--main theorem} follows immediately.
\begin{theorem}
\label{theorem--main technical theorem}
Let $(R,\m)\rightarrow (S,\n)$ be a flat local extension between local rings of equal characteristic. If $\dim R=d$, then we have $$e(R)\leq \max\{1,\frac{d!}{2^d}\}\cdot e(S).$$
\end{theorem}

This paper is organized as follows: in Section 2 we start with general preliminaries on multiplicities, and we recall some important tools that will be employed. In Section 3 we prove a technical lemma on the structure of flat local homomorphism, Lemma \ref{lemma--factoring maps with c.i. closed fibres}, which is a key ingredient in later proofs. In Section 4 we prove two results on the behavior of Hilbert-Samuel multiplicities under faithfully flat extensions of local rings of characteristic $p>0$: Theorem \ref{theorem--inequalities on multiplicities in terms of difference of embdim} and Theorem \ref{theorem--inequalities on multiplicities when the difference of embdim is large}. Theorem \ref{theorem--main technical theorem} in characteristic $p>0$ then follows immediately by combining them. Finally in Section 5, we use reduction to characteristic $p>0$ to obtain Theorem \ref{theorem--main technical theorem} in characteristic $0$.

\subsection*{Acknowledgement}
First of all, it is my great pleasure to thank Mel Hochster for introducing Lech's conjecture to me and for many enjoyable discussions. In fact, the Gorenstein case of Theorem \ref{theorem--main theorem} in characteristic $p>0$ appeared in the last section of my doctorial thesis \cite{MaThesis} written under the direction of Mel. In addition, Mel also explained to me the ideas in the reduction to characteristic $p>0$ process in Section 5.

I would like to thank Craig Huneke and Bernd Ulrich for answering my questions, for their extremely helpful comments, and for their encouragements. In particular, in discussion with Craig Huneke, we established Lemma \ref{lemma--choose generically CM sop} which eventually leads to the current proof of Theorem \ref{theorem--inequalities on multiplicities in terms of difference of embdim}, and following the comments of Bernd Ulrich, our arguments in Section 3 and Section 4 are largely simplified and shortened. I also thank Kevin Tucker and Wenliang Zhang for their comments on Theorem \ref{theorem--limit in terms of local chern characters} and related results.

The main result in dimension three in characteristic $p>0$ was first announced at the Midwest Commutative Algebra Conference at Purdue University in August 2015. At that time, our method only works in dimension three and we need some extra mild assumptions on the residue field of $R$. Since then we have largely improved our techniques to obtain Theorem \ref{theorem--main technical theorem} and thus the general version of Theorem \ref{theorem--main theorem}. %During the final preparation of this paper, we were aware of the paper \cite{SmirnovEquimultiplicityinHilbertKunztheory}, from which a more elementary proof of Theorem \ref{theorem--limit in terms of local chern characters} can be derived. We thus include two different proofs of this crucial result.

\section{Preliminaries on multiplicities}

Throughout this paper, $(R,\m)$ will always be a Noetherian local ring. In most cases, we will work with rings of equal characteristic, i.e., $R$ contains a field. We use $\nu_R(M)$ to denote the minimal number of generators of a module $M$ over $R$. If $M$ has finite length as an $R$-module, we use $l_R(M)$ to denote its length over $R$. We will drop the subscript and write $\nu(M)$, $l(M)$ when $R$ is clear from the context. We will use $\edim R$ to denote the embedding dimension of $R$, which is the same as $\nu_R(\m)$.

\subsection{Hilbert-Samuel multiplicity} For an $\m$-primary ideal $I$ of $R$ and a finitely generated $R$-module $M$, the {\it Hilbert-Samuel multiplicity} of $M$ with respect to $I$ can be defined as: $$e(I, M)=\lim_{t\to\infty}d!\cdot\frac{l_R(M/I^tM)}{t^d}$$ where $d=\dim R$. Of great importance is the case that $M=R$ and $I=\m$, where we just write $e(R)$ for the multiplicity $e(\m, R)$.

The Hilbert-Samuel multiplicity is a classical invariant that measures the singularity of $R$ (and of $M$). In general, $e(I,M)$ is always an integer and it is positive if and only if $\dim M=d$. The multiplicity $e(I, -)$ is additive on short exact sequences: if $M$ has a finite filtration by $\{M_i\}$, then $e(I, M)=\sum_ie(I, M_i)$. A rather non-trivial result is the following localization formula, which appeared in \cite{Nagata62}, and also follows from more general results on Hilbert functions \cite{Lechinequalitiesofflatcouples}, \cite{BennettOnthecharacteristicfunctionsofalocalRing}.
\begin{theorem}[localization theorem for multiplicities]
\label{theorem--localization theorem on multiplicities}
If $P$ is a prime ideal of an excellent local ring $R$ (e.g., a complete local ring $R$) such that $\dim R/P+\height P=\dim R$, then we have $e(R_P)\leq e(R)$.
\end{theorem}

The above discussion allows us to prove some reductions on Lech's conjecture. The following result is well known to experts and we include a short proof here for completeness.
\begin{lemma}
\label{lemma--reduction}
Let $(R,\m)\to (S,\n)$ be a flat local map with $\dim R=d$. In order to prove Lech's conjecture that $e(R)\leq e(S)$, or more generally, to prove $e(R)\leq C\cdot e(S)$ for certain constant $C$ depending only on $d$, it suffices to prove the case where $\dim S=\dim R=d$, $R$ and $S$ are both complete, $R$ is a domain, and $S$ has algebraically closed residue field.
\end{lemma}
\begin{proof}
As completion preserves flatness and the multiplicities, we can replace $R$ and $S$ by their completions to assume both $R$ and $S$ are complete. We can choose a minimal prime $Q$ of $\m S$ such that $\dim S/Q=\dim S/\m S$. Since $R\to S$ is faithfully flat, $\dim S=\dim S/\m S+\dim R$ and hence $$\dim S\geq \dim S/Q+\height Q\geq \dim S/\m S+\dim R=\dim S.$$ Now by Theorem \ref{theorem--localization theorem on multiplicities}, we have $e(S_Q)\leq e(S)$. Thus, if $e(R)\leq e(S_Q)$ (resp., $e(R)\leq C\cdot e(S_Q)$) then $e(R)\leq e(S)$ (resp., $e(R)\leq C\cdot e(S)$) as well. It follows that we may replace $S$ by $S_Q$ and assume $\dim S=\dim R=d$. We may lose completeness, but we can complete again. Next, we can give a filtration of $R$ by prime cyclic modules $R/P_i$, $1\leq i\leq h$. We then have $e(R)=\sum_{\dim R/P_i=d}e(R/P_i)$. After tensoring with $S$ we get a corresponding filtration of $S$ by modules $S/P_iS$, and $e(S)=\sum_{\dim S/P_iS=d}e(S/P_iS)$. Since $S$ is faithfully flat over $R$, the values of $i$ such that $\dim R/P_i=d$ are precisely those such that $\dim S/P_iS=d$. Therefore, by considering each $R/P_i\to S/P_iS$ with $\dim R/P_i=\dim R=d$, we reduce to the case that $\dim S=\dim R=d$, $R$ and $S$ are both complete, and $R$ is a domain. Finally, we can take a flat local extension $(S,\n) \to (S',\n')$ such that $\n'=\n S'$ and $S'/\n'$ is the algebraic closure of $S/\n$ (such $S'$ always exists: it is a suitable {\it gonflement} of $S$; see \cite{BourbakiBookChapter89}).\footnote{In equal characteristic, we can simply pick a coefficient field $L$ of $S$ and set $S'=S\widehat{\otimes}_L\overline{L}$.} Replacing $S$ by $S'$ and completing $S'$ if necessary, we get the desired reduction.
\end{proof}

\begin{remark}
We do not know whether Lech's conjecture can be reduced to the case that $R$ has algebraically closed (or perfect) residue field in general: we can only do this when $R$ has equal characteristic $0$ (see the proof of Lemma \ref{lemma--reduction to module finite} and Remark \ref{remark--reduction to module finite not clear in general} for further discussion on this). It is also not clear to us whether we can assume $S$ is a domain, or even reduced.
\end{remark}

If two $\m$-primary ideals $I$ and $J$ in $R$ have the same integral closure, then $e(I, M)=e(J, M)$. This is quite useful because it reduces the computation of multiplicities to the case where $I$ is a parameter ideal: when $R/\m$ is an infinite field, every $\m$-primary ideal $I$ is integral over an ideal generated by a system of parameters $\underline{x}=x_1,\dots,x_d$, which is called a {\it minimal reduction} of $I$. In particular, we have $e(I, M)=e(\underline{x}, M)$. The latter one can be computed using the Euler characteristic of the Koszul complex on $\underline{x}$: $$e(\underline{x}, M)=\chi(\underline{x}, M)=\chi(K_\bullet(\underline{x}, M))=\sum_{i=0}^d(-1)^il_R(H_i(\underline{x}, M)).$$ We will use this formula repeatedly throughout the paper. Let us also mention that one defines the higher Euler characteristic by $\chi_j(\underline{x}, M)=\sum_{i=j}^d(-1)^{j-i}l_R(H_i(\underline{x}, M))$.

\subsection{Frobenius map and the Hilbert-Kunz multiplicity} In this subsection we always assume $(R,\m)$ has equal characteristic $p>0$. We will use $R^{(e)}$ to denote the target ring of the $e$-th Frobenius map $F^e$: $R\rightarrow R$. If $M$ is an $R$-module, we will use $M^{(e)}$ to denote the corresponding module over $R^{(e)}$. We shall let $F^e_R(-)$ denote the Peskine-Szpiro's Frobenius functor from $R$-modules to $R$-modules (we will write $F^e(-)$ if $R$ is clear from the context). In detail, $F^e_R(M)$ is given by base change to $R^{(e)}$ and then identifying $R^{(e)}$ with $R$.

We say $R$ is {\it $F$-finite} if $R^{(1)}$ (or equivalently, every $R^{(e)}$) is finitely generated as an $R$-module. It is well known that, when $(R,\m)$ is complete, $R$ is $F$-finite if and only if its residue field $K=R/\m$ is $F$-finite, i.e., $[K:K^p]<\infty$. We denote $\alpha(R)=\log_p[K:K^p]$. In particular, if $(R,\m)$ has perfect residue field, then $\alpha(R)=0$. If $M$ is an $R$-module of finite length, then $l_R(M^{(e)})=p^{e\cdot\alpha(R)}l_R(M)$. A result of Kunz \cite[Proposition 2.3]{KunzOnNoetherianRingsOfCharP} shows that, when $R$ is $F$-finite, $\alpha(R_P)=\alpha(R)+\dim R/P$. %Moreover, when $R$ is $F$-finite of dimension $d$, we have $e(I,R^{(e)})=p^{e(d+\alpha(R))}\cdot e(I, R)$.because after killing a minimal prime $P$ with $\dim R/P=d$, $(R/P)^{(e)}$ always has rank equal to $p^{e(d+\alpha)}$.

For an $\m$-primary ideal $I$ of $R$ and a finitely generated $R$-module $M$, it was shown by Monsky \cite{MonskyTheHilbertKunzfunction} that the following limit $$e_{HK}(I, M)=\lim_{e\to\infty}\frac{l_R(M/I^{[p^e]}M)}{p^{ed}}$$ exists, and this is called the {\it Hilbert-Kunz multiplicity} of $M$ with respect to $I$. When $I=\m$ and $M=R$, we simplify our notation and set $e_{HK}(R)$ to be $e_{HK}(\m, R)$. It is straightforward to see that, when $(R,\m)$ is $F$-finite, we have $$e_{HK}(I,R)=\lim_{e\to\infty}\frac{l_{R}(F^e(R/I))}{p^{ed}}=\lim_{e\to\infty}\frac{l_R(R^{(e)}/IR^{(e)})}{p^{e(d+\alpha(R))}}.$$  It is worth to pointing out that if $(R,\m)$ is a local ring of dimension $d$ and $\underline{x}=x_1,\dots,x_d$ is system of parameters of $R$, then $e_{HK}(\underline{x}, R)=e(\underline{x}, R)$. However, computations of Hilbert-Kunz multiplicities of arbitrary $\m$-primary ideals have proved quite difficult. In general, we only know that if $I$ is an $\m$-primary ideal, then the multiplicities $e_{HK}(I, R)$ and $e(I, R)$ are related by the inequalities: $$\frac{e(I, R)}{d!}\leq e_{HK}(I, R)\leq e(I, R),$$ and this inequality is the best possible  in general \cite{WatanabeYoshidaHilbertKunzmultiplicity}. It is perhaps also worth remarking that, although the Hilbert-Kunz multiplicity is in general hard to study, the analog statement of Lech's conjecture for Hilbert-Kunz multiplicity is known to be true. To be more precise, if $(R,\m)\rightarrow (S,\n)$ is a flat local extension between local rings of characteristic $p>0$, then $e_{HK}(R)\leq e_{HK}(S)$ \cite{KunzOnNoetherianRingsOfCharP}, \cite{HanesThesis}. This immediately gives $$e(R)\leq d!\cdot e_{HK}(R)\leq d!\cdot e_{HK}(S)\leq d!\cdot e(S)$$ where $d=\dim R$. Thus we quickly recovered Lech's result $e(R)\leq d!\cdot e(S)$ in characteristic $p>0$ (and hence in equal characteristic, see Section 5). Our main result, Theorem \ref{theorem--main technical theorem}, improved the constant $d!$ to $\max\{1, (d!/2^d)\}$.

\subsection{Local Chern characters and Dutta multiplicity} Local Chern characters were first introduced in \cite{BaumFultonMacPhersonRiamannRochforsingularvarieties}, and a good description of them can be found in \cite[Chapter 18]{FultonIntersectiontheory}. In this paper we will need their connections with Dutta multiplicities as developed in \cite{RobertsIntersectionTheoremsMSRI}, \cite{RobertsMultiplicitiesandChernclassinLocalAlgebra}, and further extended in \cite{KuranoNumericalequivalenceonChowgroupsoflocalrings}. Below we present an outline of this theory for local rings, which will be suffices for our applications in Section 4.

Let $(R,\m)$ be a complete local ring (of arbitrary characteristic) of dimension $d$. Let $G_\bullet$ be a bounded complex of finite free $R$-modules. Let $Z$ be the support of $G_\bullet$: this is the closed subset of $\Spec R$ consisting of those primes $P$ for which the localization of $G_\bullet$ at $P$ is not exact. We denote the {\it local Chern character} of $G_\bullet$ by $\chern(G_\bullet)$. This is a sum of components:
$$\chern(G_\bullet)=\chern_0(G_\bullet)+\chern_1(G_\bullet)+\cdots+\chern_d(G_\bullet).$$
For each integer $i$, $\chern_i(G_\bullet)$ defines, for each integer $k$, a homomorphism of $\mathbb{Q}$-modules from $A_k(X)_\mathbb{Q}$ to $A_{k-i}(Z)_\mathbb{Q}$, where $A_k(X)_\mathbb{Q}$ is the $k$-th component of the Chow group with rational coefficients. These operators satisfy a lot of properties. We refer to \cite[Page 422]{RobertsIntersectionTheoremsMSRI} for a good list, \cite{RobertsMultiplicitiesandChernclassinLocalAlgebra} and \cite{FultonIntersectiontheory} for more details and proofs. Of great importance to us here is the local Riemann-Roch formula (see \cite[Section 12.6]{RobertsMultiplicitiesandChernclassinLocalAlgebra} or \cite[Example 18.3.12]{FultonIntersectiontheory} for more general versions). This formula states that for every finitely generated $R$-module $M$ there is an element $\tau(M)$ in $A_*(\Spec R)$: $$\tau(M)=[M]_d+[M]_{d-1}+\cdots+[M]_0$$ such that for every $G_\bullet$ with homology of finite length (i.e., $H_i(G_\bullet)$ is supported only at $\m$ for every $i$), we have $$\chi(G_\bullet\otimes M)=\chern_d(G_\bullet)[M]_d+\chern_{d-1}(G_\bullet)[M]_{d-1}+\cdots+\chern_0(G_\bullet)[M]_0.$$
In particular, if $G_\bullet$ is a bounded complex (of finite free $R$-modules) with homology of finite length, then
\begin{equation}
\label{equation--chi in terms of local RR}
\sum_i (-1)^il_R(H_i(G_\bullet))=\chi(G_\bullet)=\sum_{j=0}^d\chern_j(G_\bullet)[R]_j.
\end{equation}

One crucial and ingenious observation of Roberts (see \cite{RobertsIntersectionTheoremsMSRI} or \cite[Theorem 12.7.1]{RobertsMultiplicitiesandChernclassinLocalAlgebra}) is that, when $R$ has characteristic $p>0$, we have
\begin{equation}
\label{equation--dutta multiplicity in terms of chern characters}
\chi_\infty(G_\bullet):=\lim_{e\to\infty}\sum_{i}(-1)^i\frac{l_R(H_i(F^e(G_\bullet)))}{p^{ed}}=\chern_d(G_\bullet)[R]_d,
\end{equation}
where $\chi_\infty(G_\bullet)$ is called the {\it Dutta multiplicity} of the complex $G_\bullet$, which was first introduced and studied by Dutta \cite{DuttaFrobeniusandMultiplicities}. This is one of the key ingredients in the solution of the new intersection theorem in mixed characteristic.

Comparing (\ref{equation--chi in terms of local RR}) and (\ref{equation--dutta multiplicity in terms of chern characters}), we see that $\chi(G_\bullet)$ and $\chi_\infty(G_\bullet)$ differs by $\sum_{j=0}^{d-1}\chern_j(G_\bullet)[R]_j$. We say $R$ is a {\it numerically Roberts ring} if $\chi(G_\bullet)=\chi_{\infty}(G_\bullet)$ for any such $G_\bullet$. This notion was formally introduced and studied intensively in \cite{KuranoNumericalequivalenceonChowgroupsoflocalrings}.\footnote{The definition given here is taken from Theorem 6.4 in \cite{KuranoNumericalequivalenceonChowgroupsoflocalrings}, in general, one defines numerically Roberts ring in arbitrary characteristic by letting $\sum_{j=0}^{d-1}\chern_j(G_\bullet)[R]_j=0$ for every $G_\bullet$ bounded complex of finite length homology \cite[Definition 6.1]{KuranoNumericalequivalenceonChowgroupsoflocalrings}.} Let us point out that any two-dimensional equidimensional complete local ring is numerically Roberts by \cite[Example 6.6]{KuranoNumericalequivalenceonChowgroupsoflocalrings}. In higher dimension, it is well known that complete intersections are always numerically Roberts \cite{DuttaFrobeniusandMultiplicities} (see \cite[Remark 6.9]{KuranoNumericalequivalenceonChowgroupsoflocalrings}). This would be our main applications. However, we also mention that there exist Cohen-Macaulay rings of dimension three and Gorenstein rings of dimension five in characteristic $p>0$ that are not numerically Roberts \cite{RobertsIntersectionTheoremsMSRI}, \cite{SinghandMillerIntersectionmultiplicitiesoverGorensteinrings}.

Dutta multiplicities have deep connections with the Hilbert-Kunz multiplicities. We want to sketch this relation briefly. We recall the following important result, which first appeared in the main theorem of \cite{RobertsIntersectionTheoremsMSRI}. A stronger form of this theorem was also obtained in \cite[Theorem 6.2]{HochsterHunekePhantomhomology} using tight closure. Both ideas of the proofs can be traced back to \cite{DuttaFrobeniusandMultiplicities}.
\begin{theorem}
\label{theorem--vanishing of higher Koszul}
Let $(R,\m)$ be a complete local ring of characteristic $p>0$. Let $G_\bullet$ be a bounded complex of finite free $R$-modules of length $d=\dim R$ with homology of finite length. Then for every $i\geq 1$ we have $$\lim_{e\to\infty}\frac{l_R(H_i(F^e(G_\bullet)))}{p^{ed}}=0.$$
\end{theorem}

Now, suppose we have a bounded complex $G_\bullet$ of length exactly $d=\dim R$, with homology of finite length. It then follows from Theorem \ref{theorem--vanishing of higher Koszul} that $$\chi_\infty(G_\bullet)=\lim_{e\to\infty}\frac{l_R(H_0(F^e(G_\bullet)))}{p^{ed}}.$$ Therefore, if we also have $H_0(G_\bullet)=R/I$, then $\chi_\infty(G_\bullet)=e_{HK}(I, R)$. In particular, if $(R,\m)$ is a numerically Roberts ring and $I$ is an $\m$-primary ideal of $R$ of finite projective dimension (this implies $R$ is Cohen-Macaulay by the new intersection theorem \cite{RobertsIntersectionTheoremsMSRI}), then $$\lim_{e\to\infty}\frac{l(R/I^{[p^e]})}{p^{e\cdot \dim R}}=l(R/I).$$ We will use similar ideas repeatedly in Section 4. We also refer to \cite[Section 6]{KuranoNumericalequivalenceonChowgroupsoflocalrings} for more general results of this type.

\section{Structure of flat local maps}
Our goal in this section is to prove Lemma \ref{lemma--factoring maps with c.i. closed fibres}, which will be used in Section 4. We first recall the main theorem from \cite{AvramovFoxbyHerzogStructureoflocalmap}, which can be viewed as a natural generalization of Cohen's structure theorem for complete local rings.
\begin{theorem}[Cohen factorizations]
\label{theorem--Cohen factorization}
A local homomorphism $(R,\m)\rightarrow (S,\n)$ with $S$ complete can be factored as $(R,\m)\rightarrow (T,\n_T)\rightarrow (S,\n)$ such that $(R,\m)\rightarrow (T,\n_T)$ is flat local with $T/\m T$ regular, $(T,\n_T)$ is complete, and $(T,\n_T)\rightarrow (S,\n)$ is surjective.

Moreover, if $S$ has finite flat dimension over $R$ (e.g., $S$ is flat over $R$), then $S$ has finite projective dimension over $T$.
\end{theorem}

A {\it Cohen-factorization} as in Theorem \ref{theorem--Cohen factorization} is called {\it minimal} if, with $S=T/J$, we have $J\subseteq \m T+\n_T^2$. This can be reduced from any Cohen-factorization by killing part of a regular system of parameters on $T/\m T$ that is contained in $J$ but not contained in $\n_T^2(T/\m T)$ \cite[Proposition 1.5]{AvramovFoxbyHerzogStructureoflocalmap}. However in this paper, to get the desired inequalities on multiplicities, we need to factor the map $(R,\m)\to (T, \n_T) \to (S,\n)=T/J$ such that $\pd_TS<\infty$ and $J\subseteq \n_T^2$. We cannot always achieve this while keeping $T/\m T$ regular. Our key observation here is that, at least for flat local extensions of rings of the same dimension, we can achieve this at the expense of letting $T/\m T$ be a complete intersection. To establish such a factorization we first recall two classical and crucial results.

\begin{theorem}[Section 21 of \cite{Matsumura70}]
\label{theorem--flat map with regular fibre} Let $(R,\m)\rightarrow (T,\n_T)$ be a flat local map. If $x_1,\dots,x_t$ is a regular sequence in $T/\m T$, then it is a regular sequence on $T/IT$ for every $I\subseteq R$, and $T/(x_1,\dots,x_t)T$ is faithfully flat over $R$.
\end{theorem}

\begin{lemma}[Theorem 27.3 of \cite{Nagata62}]
\label{lemma--killing linear term preserves finite projective dimension}
Let $(T,\n)$ be a local ring and $M$ be a finitely generated $T$-module such that $\pd_TM<\infty$. If $x\in \Ann_RM$ such that $x\notin \m^2\cup(\cup_{P\in\Ass(T)}P)$, then $\pd_{T/xT}M<\infty$.
\end{lemma}

It should be pointed out that Lech proved in \cite{Lechinequalitiesofflatcouples} that for every flat local extension $(R,\m)\to (S,\n)$, we always have $\edim R-\dim R\leq \edim S-\dim S$. In particular, we have $\edim S-\edim R\geq 0$ (note that in general, $\edim S-\edim R$ is not the same as $\edim (S/\m S)$). Below we give a short proof of a more general fact regarding local maps of finite flat dimension. This also answers a question in \cite[Remark 6.3]{AvramovFoxbyHalperinDecentandascent}.

\begin{theorem}
\label{theorem--regularity defect of flat local extension}
Let $(R,\m)\rightarrow (S,\n)$ be a local map such that $S$ has finite flat dimension over $R$ (e.g., $S$ is flat over $R$), then we have
\begin{equation*}
\edim R-\dim R\leq \edim S-\dim S.
\end{equation*}
In particular, if $\dim S\geq \dim R$ (e.g., $S$ is flat over $R$), then $\edim S-\edim R\geq 0$.
\end{theorem}
\begin{proof}
We may assume both $R$ and $S$ are complete. By Theorem \ref{theorem--Cohen factorization}, we have $$(R,\m)\rightarrow (T,\n_T) \rightarrow (S,\n)$$ such that $R\rightarrow T$ is faithfully flat with $T/\m T$ regular, and $S=T/J$ with $\pd_TS<\infty$. It is easy to see that $\edim R-\dim R=\edim T-\dim T$, so it suffices to prove $\edim T-\dim T\leq \edim S-\dim S$ when $\pd_TS<\infty$. We claim that:
\begin{equation}
\label{equation--inequality of embedding dimension and depth}
\edim T\leq \edim S+\depth_JT.
\end{equation}

This is easy if $J\subseteq \n_T ^2$, because in this case we have $\edim S=\edim T$, so (\ref{equation--inequality of embedding dimension and depth}) holds trivially. Now assume $0\neq J\nsubseteq \n_T ^2$. Since $\pd_TS<\infty$, we know that $J$ contains a nonzerodivisor of $T$ \cite[Corollary 20.13]{Eisenbud95}. Thus by prime avoidance, there exists $x\in J$ such that $x\notin \n_T ^2 \cup(\cup_{P\in\Ass T} P)$. Let $\overline{T}=T/xT$ and $\overline{J}=J\overline{T}$. We still have $S=\overline{T}/\overline{J}$ with $\pd_{\overline{T}}S<\infty$ by Lemma \ref{lemma--killing linear term preserves finite projective dimension}. Moreover, $\edim T$ drops by one while $\edim S$ stays the same. But we can do this process at most $\depth_JT$ times (we either end up with $J=0$ or we stop at some point with $J\subseteq \n_T^2$), thus (\ref{equation--inequality of embedding dimension and depth}) follows.

Since we always have $\depth_JT \leq \dim T-\dim T/J=\dim T- \dim S$. Combining this with (\ref{equation--inequality of embedding dimension and depth}) we thus get $\edim T-\dim T\leq \edim S-\dim S$. This finishes the proof.
\end{proof}

Recall that for a local ring $T$ and an ideal $J\subseteq T$, we always have
\begin{equation}
\label{equation--inequality on depth, height and projdim}
\depth_JT\leq \height{J}\leq \dim T-\dim (T/J)\leq \pd_T(T/J).
\end{equation}
The first two inequalities are trivial, while the third inequality is a consequence of the celebrated new intersection theorem \cite{RobertsIntersectionTheoremsMSRI}. An ideal $J\subseteq T$ is called {\it perfect} if $\pd_T(T/J)=\depth_JT$, and hence for perfect ideals all the inequalities in (\ref{equation--inequality on depth, height and projdim}) are equalities.

We next prove the following:

\begin{lemma}
\label{lemma--perfect ideal}
Let $(R,\m)\to (S,\n)$ be a flat local map between complete local rings of the same dimension. Suppose we have a factorization $$(R,\m)\to (T,\n_T)\to (S,\n)$$ such that $R\to T$ is flat local, and $S=T/J$ with $\pd_TS<\infty$.  Then $J$ is a perfect ideal in $T$.
\end{lemma}
\begin{proof}
Since $R\to S$ is flat local with $\dim R=\dim S$, we know that $\m S$ is $\n$-primary and $\depth R=\depth S$. By the Auslander-Buchsbaum formula, we have $$\pd_TS+\depth S=\depth T=\depth R+\depth T/\m T.$$
Therefore we get $\pd_TS=\depth T/\m T$. Since $\m S$ is $\n$-primary, $\m T+J$ is $\n_T$-primary. Thus we know that $J\cdot(T/\m T)=(J+\m T)/\m T$ is $\n_T$-primary in $T/\m T$. Hence we can pick $y_1,\dots, y_n \in J$ such that $y_1,\dots,y_n$ form a regular sequence on $T/\m T$ with $n=\depth T/\m T$. By Theorem \ref{theorem--flat map with regular fibre}, $y_1,\dots, y_n$ is a regular sequence on $T$. This implies $$\depth_JT\geq n=\depth T/\m T=\pd_TS.$$ Since the other direction always holds by (\ref{equation--inequality on depth, height and projdim}), we have $\pd_TS=\depth_JT$ and thus $J$ is perfect.
\end{proof}

We are ready to state and prove our lemma on factoring flat local maps that will be used in Section 4. This lemma is also of independent interest.

\begin{lemma}
\label{lemma--factoring maps with c.i. closed fibres}
Let $(R,\m)\to (S,\n)$ be a flat local map between complete local rings of the same dimension. Suppose $\edim S-\edim R=c$. Then this map can be factored as $$(R,\m)\rightarrow (T,\n_{T})\rightarrow (S,\n)=T/J$$ such that the following are satisfied:
\begin{enumerate}
\item $(R,\m)\rightarrow (T,\n_{T})$ is flat local with $(T,\n_T)$ complete and $T/\m T$ a complete intersection;
\item $J$ is a perfect ideal and $\pd_TS=c$;
\item $J\subseteq \n_T^2$.
\end{enumerate}
\end{lemma}
\begin{proof}
We first note that $c\geq 0$ by Theorem \ref{theorem--regularity defect of flat local extension}. Now we let: $$(R,\m)\to (T',\n_{T'})\to (S,\n)$$ be a Cohen-factorization as in Theorem \ref{theorem--Cohen factorization}, where $R\to T'$ is flat local with $T'/\m T'$ regular, and $S=T'/J'$ with $\pd_{T'}S<\infty$. Suppose $\dim T'/\m T'=\depth T'/\m T'=b$. Since $S$ is a quotient of $T'$, $\edim S \leq \edim T'$, and we have $$b=\edim T'-\edim R\geq\edim S-\edim R=c \geq 0.$$

If $b=c$ then $\edim S=\edim T'$ and hence $J\subseteq \n_{T'}^2$. Thus we simply set $T=T'$ and $J=J'$ and one can check that (1)--(3) are all satisfied by Theorem \ref{theorem--Cohen factorization} and Lemma \ref{lemma--perfect ideal}.

Now suppose $b>c$, we claim that there exists $y_1,\dots,y_{b-c}\in J'$ satisfying the following conditions:
\begin{enumerate}[ (a)\,\,]
\item The image of $y_1,\dots,y_{b-c}$ in $T'/\m T'$ form a regular sequence on $T'/\m T'$;
\item The image of $y_1,\dots,y_{b-c}$ in $\n_{T'}/\n_{T'}^2$ form part of a basis for $\n_{T'}/\n_{T'}^2$.
\end{enumerate}

We construct these elements inductively: %since $\edim T'- \edim S=b-c>0$, $J'\nsubseteq\n_{T'}^2$. By prime avoidance there exists $y_1\in J'$ such that $y_1\notin \n_{T'}^2\cup\m T'$. Thus (a) holds because $T'/\m T'$ is regular and $y_1\notin\m T'$ and (b) holds since $y_1\notin\n_{T'}^2$. Next we
suppose we already find $y_1,\dots,y_j$, $0\leq j<b-c$ ($j=0$ is the initial case). Let $\overline{T}=T'/(y_1,\dots,y_j)T'$ and $\overline{J}=J'\overline{T}$.  Since $y_1,\dots,y_j\in J'$, we still have $S=\overline{T}/\overline{J}$. Because $\edim \overline{T}=\edim T'-j>\edim T'-(b-c)=\edim S$,  $\overline{J}\nsubseteq\n_{\overline{T}}^2$. Because $(R,\m)\to (S,\n)$ is flat local with $\dim R=\dim S$, $\m S$ is $\n$-primary, which implies $\m \overline{T}+\overline{J}$ is $\n_{\overline{T}}$-primary since $S=\overline{T}/\overline{J}$. But by the inductive hypothesis $\overline{T}/\m\overline{T}$ is a complete intersection of dimension $b-j>0$. In particular, every associated prime $P$ of $\overline{T}/\m \overline{T}$ has $\dim \overline{T}/P=b-j>0$. This implies $\m\overline{T}+\overline{J}\nsubseteq P$ and hence $\overline{J}\nsubseteq P$ for every associated prime $P$ of $\overline{T}/\m\overline{T}$. Now by prime avoidance, we have: $$\overline{J}\nsubseteq \n_{\overline{T}}^2\cup (\cup_{P\in\Ass\overline{T}/\m \overline{T}}P).$$ Therefore we can pick $y_{j+1}\in \overline{J}$ such that $y_{j+1}\notin \cup_{P\in\Ass\overline{T}/\m \overline{T}}P$ and $y_{j+1}\notin \n_{\overline{T}}^2$. But this precisely means $y_1,\dots,y_{j+1}$ satisfies (a) and (b). This finishes the proof of our claim.

We set $T=T'/(y_1,\dots,y_{b-c})$ and we claim that $T$ satisfies (1)--(3). It follows directly from condition (a) and Theorem \ref{theorem--flat map with regular fibre} that $R\to T$ is flat local with $T$ complete and $T/\m T$ a complete intersection, and thus $T$ satisfies (1). Clearly we have $S=T/J$ where $J=J'T$. Since $y_1,\dots,y_{b-c}$ is a regular sequence on $T'/\m T'$ by (a), Theorem \ref{theorem--flat map with regular fibre} tells us that $y_1,\dots,y_{b-c}$ is a regular sequence in $T'$. This together with (b) implies $\pd_{T}S<\infty$ by Lemma \ref{lemma--killing linear term preserves finite projective dimension}. Hence by Lemma \ref{lemma--perfect ideal}, $J$ is a perfect ideal of $T$. By the Auslander-Buchsbaum formula,
$$\pd_{T}S=\depth T-\depth S=\depth T-\depth R.$$  But since $T$ is obtained from $T'$ by killing a regular sequence of length $b-c$, we have  $$\pd_{T}S=\depth T'-\depth R-(b-c)=\depth T'/\m T'-(b-c)=c,$$ which verified (2). Finally, because $y_1,\dots,y_{b-c}$ is part of a basis for $\n_{T'}/\n_{T'}^2$ by (b), $$\edim T=\edim T'-(b-c)=\edim R+c=\edim S,$$ which implies $J\subseteq\n_{T}^2$. Therefore we have verified (3) and hence finished the proof of the lemma.
\end{proof}

Lemma \ref{lemma--factoring maps with c.i. closed fibres} immediately implies the following corollary, which was originally proved in \cite{Lechinequalitiesofflatcouples} using different methods.

\begin{corollary}
\label{corollary--difference of embedding dimension <=1}
Let $(R,\m)\to (S,\n)$ be a flat local map between complete local rings of the same dimension. If $\edim S-\edim R\leq 1$, then $S/\m S$ is a complete intersection.
\end{corollary}
\begin{proof}
We factor this map as in Lemma \ref{lemma--factoring maps with c.i. closed fibres}. By Lemma \ref{lemma--factoring maps with c.i. closed fibres} (2), $\pd_TS=\pd_T(T/J)\leq 1$. So $J$ is either $0$ or a principal ideal generated by a nonzerodivisor $y$ in $T$. In the latter case we claim that $y$ must be a nonzerodivisor on $T/\m T$ also. This is because $$\dim (T/\m T)/y(T/\m T)=\dim S/\m S=0,$$ while $T/\m T$ is a complete intersection with $$\dim T/\m T=\dim T-\dim R=\dim T-\dim S=1.$$ Therefore in either case, $S/\m S$ is a ($0$-dimensional) complete intersection.
\end{proof}

We end this section by proving a lemma on the behavior of Hilbert-Kunz multiplicities under flat local map with regular closed fiber. We believe this result (and perhaps more general results) are well known to experts. But we could not find a reference that can cover the generality we need, so we give the proof.

\begin{lemma}
\label{lemma--Appendix}
Let $(R,\m)\to (T,\n_T)$ be a flat local extension between complete local rings of characteristic $p>0$ such that $T/\m T$ is regular. Then
\begin{enumerate}
\item $e_{HK}(T)=e_{HK}(R);$
\item $e_{HK}(\n_T^2, T)=e_{HK}(\m^2, R)+ (\edim T-\edim R)\cdot e_{HK}(R).$
\end{enumerate}
\end{lemma}
\begin{proof}
Let $\dim T/\m T=\edim T-\edim R=n$ and let $x_1,\dots,x_n$ be a regular system of parameters in $T/\m T$. We set $(T_0,\n_0)=R[[x_1,\dots,x_n]]$. Note that $(T_0,\n_0)\to (T, \n_T)$ is a flat local map such that the closed fiber $T/\n_0T$ is a field, i.e., $\n_0T=\n_T$. Hence it is clear that $$e_{HK}(T)=e_{HK}(T_0)=e_{HK}(R) \text{ and } e_{HK}(\n_T^2, T)=e_{HK}(\n_0^2, T_0).$$ It thus remains to show that $$e_{HK}(\n_0^2, T_0)=e_{HK}(\m^2, R)+ (\edim T_0-\edim R)\cdot e_{HK}(R).$$ Since $T_0$ is a power series over $R$ with $n=\edim T_0-\edim R$ variables, by an easy induction it suffices to prove that $$e_{HK}((\m+x)^2, R[[x]])=e_{HK}(\m^2, R)+e_{HK}(R).$$ Now we observe that
$$l\left(\frac{R[[x]]}{((\m+x)^2)^{[p^e]}}\right)=l\left(\frac{R[[x]]}{(\m^2)^{[p^e]}+\m^{[p^e]}x^{p^e}+(x^2)^{[p^e]}}\right)=p^e\cdot l(R/(\m^2)^{[p^e]})+p^e\cdot l(R/\m^{[p^e]}).$$
Dividing both sides by $p^{e(\dim R+1)}$ and taking limit we find that $e_{HK}((\m+x)^2, R[[x]])=e_{HK}(\m^2, R)+e_{HK}(R)$, as desired.
\end{proof}

\section{The main result in characteristic $p>0$}

In this section we will prove Theorem \ref{theorem--main technical theorem} in equal characteristic $p>0$. Our proof heavily uses Hilbert-Kunz theory. We will give two inequalities on the multiplicities from which Theorem \ref{theorem--main technical theorem} follows immediately. Throughout this section, we assume all rings have equal characteristic $p>0$.

\subsection{An inequality on multiplicities in terms of $\edim S-\edim R$} We begin by proving the following lemma that is a consequence of Theorem \ref{theorem--vanishing of higher Koszul}.
\begin{lemma}
\label{lemma--vanishing of higher Koszul}
%Let $(R,\m)\to (S,\n)$ be a flat local extension between complete local rings of characteristic $p>0$ with $\dim R=\dim S=d$. Suppose we have a factorization $$(R,\m)\to(T,\n_T)\to (S,\n)=T/J$$ such that $R\to T$ is flat local and $\pd_TS<\infty$ (for example, we can take a Cohen-factorization as in Theorem \ref{theorem--Cohen factorization}, or any factorization as in Lemma \ref{lemma--factoring maps with c.i. closed fibres}).
Let $(T,\n_T)$ be a complete local ring of characteristic $p>0$ with $T/\n_T$ a perfect field. Let $J$ be a perfect ideal of $T$ and $(S,\n)=T/J$. Then for every system of parameters $\underline{x}=x_1,\dots,x_d$ of $S$, we have $$\lim_{e\to\infty}\frac{l_S(H_i(\underline{x}, T^{(e)}\otimes S))}{p^{e\cdot\dim T}}=0$$ for every $i\geq 1$.
\end{lemma}
\begin{proof}
Since $J$ is a perfect ideal, we have $\pd_TS=\dim T-\dim S$. Let $G_\bullet$ be a minimal free resolution of $S$ over $T$ and let $K_\bullet(\underline{x}, T)$ be the Koszul complex on $\underline{x}$. We have $$H_i(\underline{x}, T^{(e)}\otimes S)=H_i(K_\bullet(\underline{x}, T)\otimes T^{(e)}\otimes S)=H_i(T^{(e)}\otimes K_\bullet(\underline{x}, T)\otimes G_\bullet).$$
Next we note that $K_\bullet(\underline{x}, T)\otimes G_\bullet$ is a complex of free $T$-modules of length $\dim S+\pd_TS=\dim T$, with homology of finite length. Hence by Theorem \ref{theorem--vanishing of higher Koszul}, $$\lim_{e\to\infty}\frac{l_T\left(H_i(F_T^e(K_\bullet(\underline{x}, T)\otimes G_\bullet))\right)}{p^{e\cdot\dim T}}=0$$ for every $i\geq 1$.
But since $T/\n_T$ is a perfect field, we have $$l_S(H_i(\underline{x}, T^{(e)}\otimes S))=l_T(H_i(T^{(e)}\otimes K_\bullet(\underline{x}, T)\otimes G_\bullet))=l_T\left(H_i(F_T^e(K_\bullet(\underline{x}, T)\otimes G_\bullet)\right). \qedhere$$
\end{proof}

Before we proceed, we emphasize that if $(T,\n_T)$ is a complete local ring of characteristic $p>0$ such that $T/\n_T$ is a perfect field, then $T^{(e)}$ is a finitely generated $T$-module (i.e., $T$ and hence all localizations of $T$ are $F$-finite). Because a complete local ring is $F$-finite if and only if its residue field is $F$-finite. Therefore $T^{(e)}\otimes_TS$ is a finitely generated $S$-module for every map $T\to S$, and in particular, we can talk about the multiplicities of $T^{(e)}\otimes S$.

\begin{lemma}
\label{lemma--multiplicity for perfect ideal generically numerically Roberts}
Let $(T,\n_T)$ be a complete local ring of characteristic $p>0$ with $T/\n_T$ a perfect field. Let $J$ be a perfect ideal of $T$ and $S=T/J$. Then for every system of parameters $\underline{x}=x_1,\dots,x_d$ of $S$, we have
\begin{equation}
\label{equation--equality on multiplicities and the crucial limit}
e_{HK}(J+\underline{x}, T)=\lim_{e\to\infty}\frac{1}{p^{e\cdot \dim T}}\cdot e(\underline{x}, T^{(e)}\otimes_TS)=\sum_Pe(\underline{x}, T/P)e_{HK}(J, T_P)
\end{equation}
where the sum is taken over all minimal prime $P$ of $J$ such that $\dim T/P=\dim S$.

Moreover, if $T_P$ is a numerically Roberts ring (e.g., $T_P$ is a complete intersection) for every such $P$, then the above is also equal to $e(\underline{x}, S)$.
\end{lemma}
\begin{proof}
By the Koszul characterization of multiplicity, we know that
$$\lim_{e\to\infty}\frac{1}{p^{e\cdot \dim T}}\cdot e(\underline{x}, T^{(e)}\otimes_TS)=\lim_{e\to\infty}\sum_i(-1)^i\cdot \frac{l_S\left(H_i(\underline{x}, T^{(e)}\otimes S)\right)}{p^{e\cdot\dim T}}.$$
Since $J$ is a perfect ideal by Lemma \ref{lemma--perfect ideal}, Lemma \ref{lemma--vanishing of higher Koszul} tells us all the higher terms on the right hand side vanish, hence the above is equal to
$$\lim_{e\to\infty}\frac{l_S\left(H_0(\underline{x}, T^{(e)}\otimes S)\right)}{p^{e\cdot\dim T}}=\lim_{e\to\infty}\frac{l_T\left(F^e_T(T/((\underline{x})+J))\right)}{p^{e\cdot\dim T}}=e_{HK}(J+(\underline{x}), T).$$
This proves the first equality in (\ref{equation--equality on multiplicities and the crucial limit}).

Next, suppose $P$ is a minimal prime of $J$ such that $\dim T/P=\dim S$. Because $T/\n_T$ is perfect, we have $$\alpha(T_{P})=\alpha(T)+\dim T/P=\dim T/P.$$ Since $J$ is a perfect ideal, we have $\pd_TS=\height J=\dim T-\dim S$, thus by our choice of $P$, $$\dim T\geq \dim T/P+\dim T_{P}\geq \dim S+\height J=\dim T.$$ Hence we have $\dim T_{P}+\alpha(P)=\dim T$ for every minimal prime $P$ of $J$ such that $\dim T/P=\dim S$. Now by the associativity formula for multiplicities \cite[Theorem 11.2.4]{HunekeSwansonIntegralClosure}, we have
\begin{eqnarray*}
\frac{1}{p^{e\cdot\dim T}}\cdot e(\underline{x}, T^{(e)}\otimes_TS)&=& \frac{1}{p^{e\cdot\dim T}}\cdot\sum_Pl_{T_{P}}\left((T^{(e)}\otimes_T S)_{P}\right)\cdot e(\underline{x}, T/P)\\
&=&\frac{1}{p^{e\cdot\dim T}}\cdot\sum_Pl_{T_{P}}\left((T_{P})^{(e)}/J(T_{P})^{(e)}\right)\cdot e(\underline{x}, T/P)\\
&=&\frac{1}{p^{e\cdot\dim T}}\cdot p^{e\cdot\alpha(T_{P})}\sum_Pl_{T_{P}}(T_{P}/J^{[p^e]}T_{P})\cdot e(\underline{x}, T/P)\\
&=&\sum_P\frac{1}{p^{e\cdot\dim T_{P}}}l_{T_{P}}(T_{P}/J^{[p^e]}T_{P})\cdot e(\underline{x}, T/P).
\end{eqnarray*}
Now we take the limit as $e\to\infty$ and by the definition of Hilbert-Kunz multiplicity, we have
$$\lim_{e\to\infty}\frac{1}{p^{e\cdot \dim T}}\cdot e(\underline{x}, T^{(e)}\otimes_TS)=\sum_Pe(\underline{x}, T/P)e_{HK}(J, T_P).$$ This proves the second equality in (\ref{equation--equality on multiplicities and the crucial limit}).

Finally, if $T_{P}$ is a numerically Roberts ring, then we know that $e_{HK}(J, T_P)=l_{T_{P}}(T_{P}/JT_{P})$ because $JT_P$ is a $PT_P$-primary ideal of finite projective dimension. So we have:
$$\sum_Pe(\underline{x}, T/P)e_{HK}(J, T_P)=\sum_Pe(\underline{x}, T/P)l_{T_{P}}(T_{P}/JT_{P})=e(\underline{x}, T/J)=e(\underline{x}, S). \qedhere$$
\end{proof}

\begin{remark}
It is worth to mention that the last assertion of Lemma \ref{lemma--multiplicity for perfect ideal generically numerically Roberts} is {\it false} in general if we do not assume $T_P$ is numerically Roberts. In \cite{RobertsIntersectionTheoremsMSRI}, based on earlier work of \cite{DuttaHochsterMcLaughlinModulesoffiniteprojectivedimension}, Roberts constructed an example of a three-dimensional Cohen-Macaulay ring $(T,\n_T)$ and an $\n_T$-primary ideal $J$ of finite projective dimension such that $$e_{HK}(J, T)=\lim_{e\to\infty}\frac{l_T(T/J^{[p^e]})}{p^{e\cdot\dim T}}\neq l_T(T/J).$$
Therefore we can set $S=T/J$ (and $\underline{x}$ to be the empty system of parameters, i.e., the zero ideal) in Lemma \ref{lemma--multiplicity for perfect ideal generically numerically Roberts} to see that the desired equality fails: for an Artinian local ring, the multiplicity of any ideal is equal to the length of the ring. Of course, the problem is that our $T$ is not numerically Roberts.
\end{remark}

In order to apply Lemma \ref{lemma--multiplicity for perfect ideal generically numerically Roberts} in our setting, we need the following celebrated result on the localization problem of Grothendieck, see \cite{TabaaCompleteintersectionhomomorphism} or \cite[Theorem 4.1]{AvramovFoxbyGrothendiecklocalizationproblem}:

\begin{theorem}
\label{theorem--flat map with complete intersection closed fibers}
Let $\varphi$: $(R,\m)\to (T,\n_T)$ be a flat local map with $R$ complete. If the closed fiber $T/\m T$ is a complete intersection, then all fibers of $\varphi$ are complete intersections.
\end{theorem}

We are ready to prove the following:

\begin{lemma}
\label{lemma--multiplicity in terms of the limit}
Let $(R,\m)\to (S,\n)$ be a flat local map between complete local rings of characteristic $p>0$ with $\dim R=\dim S$. Suppose we have a factorization $$(R,\m)\to (T,\n_T)\to (S,\n)=T/J$$ such that $R\to T$ is flat local with $T/\m T$ a complete intersection and $\pd_TS<\infty$ (for example, we can take a Cohen-factorization, or any factorization as in Lemma \ref{lemma--factoring maps with c.i. closed fibres}). If $R$ is a domain and $S$ has perfect residue field, then $$e_{HK}(J+\underline{x}, T)=\lim_{e\to\infty}\frac{1}{p^{e\cdot \dim T}}\cdot e(\underline{x}, T^{(e)}\otimes_TS)=e(\underline{x}, S)$$ for every system of parameters $\underline{x}$ of $S$.
\end{lemma}
\begin{proof}
By Lemma \ref{lemma--perfect ideal}, we know $J$ is a perfect ideal. Let $P_1,\dots,P_m$ be the minimal primes of $J$. Since $R\to S=T/J$ is faithfully flat, we know $P_i\cap R=0$ because we assumed $R$ is a domain. Since $T/\m T$ is a complete intersection, by Theorem \ref{theorem--flat map with complete intersection closed fibers}, $T_{P_i}/(P_i\cap R)T_{P_i}=T_{P_i}$ is a complete intersection and hence numerically Roberts. Therefore all the hypotheses of Lemma \ref{lemma--multiplicity for perfect ideal generically numerically Roberts} are satisfied and the result follows.
\end{proof}

If we choose $\underline{x}$ to be a minimal reduction of $\n$ in Lemma \ref{lemma--multiplicity in terms of the limit}, then we see immediately that $e(S)=e(\underline{x}, S)$ is equal to the Hilbert-Kunz multiplicity $e_{HK}(J+\underline{x}, T)$. However, in general, the dimension of $T$ is large and we don't have control on the ideal $J+(\underline{x})$ either. So $e_{HK}(J+(\underline{x}), T)$ will not give us a good estimate of $e(S)$ in terms of the Hilbert-Kunz multiplicity of $R$. Motivated by Roberts's work \cite{RobertsIntersectionTheoremsMSRI}, \cite{RobertsMultiplicitiesandChernclassinLocalAlgebra}, in the next theorem we analyze $\lim_{e\to\infty}\frac{1}{p^{e\cdot \dim T}}\cdot e(\underline{x}, T^{(e)}\otimes_TS)$ in more details using local Chern characters. As a consequence we will see that for certain carefully chosen $\underline{x}=x_1,\dots, x_d$, we will have $$e(S)=e(\underline{x}, S)=e_{HK}(J, T/(\underline{x})T).$$ It turns out that this is crucial to get our desired estimate on $e(S)$ once we choose our factorization as in Lemma \ref{lemma--factoring maps with c.i. closed fibres}.

During the preparation of this paper, we were aware of a beautiful result \cite[Corollary 4.11]{SmirnovEquimultiplicityinHilbertKunztheory}. Based on this result and our Lemma \ref{lemma--multiplicity for perfect ideal generically numerically Roberts}, we are also able to give a completely elementary proof of the next theorem which could avoid the use of local Chern characters. Since we feel both proofs reveal some nature about the limit of multiplicities, we decide to keep both proofs here. %The readers who are not familiar with local Chern characters can skip the first proof.

%can be viewed as a generalization of \cite[Theorem 12.7.1]{RobertsMultiplicitiesandChernclassinLocalAlgebra}.
%$$\lim_{e\to\infty}\frac{1}{p^{e\cdot \dim T}}\cdot e(\underline{x}, T^{(e)}\otimes_TS)=\chern_n(\overline{G}_\bullet)\sum_j (-1)^j[H_j(\underline{x}, T)]_n$$
%where $\overline{G}_\bullet$ denote the the complex $G_\bullet\otimes_TT/(\underline{x})T$.

\begin{theorem}
\label{theorem--limit in terms of local chern characters}
Let $(T,\n_T)$ be a complete local ring of characteristic $p>0$ with $T/\n_T$ a perfect field. Let $J\subseteq T$ be a perfect ideal and $S=T/J$. Set $n=\height J$ and $d=\dim T/J$ (thus $\dim T=n+d$), and let $\underline{x}=x_1,\dots,x_d$ be a system of parameters of $S$ that is also part of a system of parameters of $T$. Then we have
\begin{equation}
\label{equation--interchanging two limits}
\lim_{e\to\infty}\frac{1}{p^{e\cdot \dim T}}\cdot e(\underline{x}, T^{(e)}\otimes_TS)=\sum_Pe_{HK}(J, T/P)e(\underline{x}, T_P)
\end{equation}
where the sum is taken over all minimal primes $P$ of $(x_1,\dots,x_d)$ of dimension $n$.

In particular, if $T_P$ is Cohen-Macaulay for every such $P$, then we have
$$\lim_{e\to\infty}\frac{1}{p^{e\cdot \dim T}}\cdot e(\underline{x}, T^{(e)}\otimes_TS)=e_{HK}(J, T/(\underline{x})T).$$
\end{theorem}
\begin{proof}
We first note that, if $T_P$ is Cohen-Macaulay, then $e(\underline{x}, T_P)=l_{T_P}(T_P/(\underline{x})T_P)$. Therefore the second assertion follows immediately from the first one by the associativity formula for Hilbert-Kunz multiplicities. Thus below we aim to prove (\ref{equation--interchanging two limits}).

\vspace{1em}

\noindent{\it\textbf{First proof of (\ref{equation--interchanging two limits})}}: Fix a minimal free resolution $G_\bullet$ of $T/J$ over $T$. Let $K_\bullet=K_\bullet(\underline{x}, T)$ be the Koszul complex with respect to $\underline{x}=x_1,\dots,x_d$. By the Koszul characterization of multiplicity, we have
\begin{eqnarray*}
\frac{1}{p^{e\cdot \dim T}}\cdot e(\underline{x}, T^{(e)}\otimes_TS)&=&\frac{1}{p^{e\cdot \dim T}}\cdot\chi(T^{(e)}\otimes K_\bullet\otimes G_\bullet)\\
&=&\frac{1}{p^{e(n+d)}}\sum_{j=0}^{n+d}\chern_j(F^e_T(K_\bullet\otimes G_\bullet))[T]_j\\
&=&\frac{1}{p^{e(n+d)}}\sum_{j=0}^{n+d}\chern_j(K_\bullet\otimes G_\bullet)[F^e_*T]_j\\
&=&\frac{1}{p^{e(n+d)}}\sum_{j=0}^{n+d}p^{ej}\chern_j(K_\bullet\otimes G_\bullet)[T]_j
\end{eqnarray*}
where the equality on the second line is by the local Riemann-Roch formula and the equality on the third line is by the projection formula (note that we also used repeatedly here that $T/\n_T$ is perfect).

Next we observe that, when $e\to\infty$, the only term in $\frac{1}{p^{e(n+d)}}\sum_{j=0}^{n+d}p^{ej}\chern_j(K_\bullet\otimes G_\bullet)[T]_j$ that can survive is the top term, i.e., when $j=n+d$. Hence we have:
$$\lim_{e\to\infty}\frac{1}{p^{e\cdot \dim T}}\cdot e(\underline{x}, T^{(e)}\otimes_TS)=\chern_{n+d}(K_\bullet\otimes G_\bullet)[T]_{n+d}=\sum_j\chern_{n+d-j}(G_\bullet)\chern_{j}(K_\bullet)[T]_{n+d}$$
Since $\chern(K_\bullet)\eta=\chern_d(K_\bullet)\eta=(x_1)\cap(x_2)\cap\cdots\cap(x_d)\cap \eta$ for every cycle $\eta$ (for example, see \cite[Corollary 12.3.2]{RobertsMultiplicitiesandChernclassinLocalAlgebra}) and the computation of intersection with divisors can be explicitly expressed using Koszul homologies \cite[Proposition 5.2.11]{RobertsMultiplicitiesandChernclassinLocalAlgebra}, we have
\begin{eqnarray}
\label{equation--intersection with divisors}
\lim_{e\to\infty}\frac{1}{p^{e\cdot \dim T}}\cdot e(\underline{x}, T^{(e)}\otimes_TS)&=&\chern_n(G_\bullet)\chern_d(K_\bullet)[T]_{n+d}\\
&=&\chern_n(G_\bullet)((x_1)\cap\cdots\cap(x_d)\cap[T]_{n+d})\notag\\
&=&\chern_n(G_\bullet)\sum_j (-1)^j[H_j(\underline{x}, T)]_n\notag\\
&=&\chern_n(G_\bullet)\sum_P \chi(\underline{x}, T_P)[T/P]_n\notag\\
&=&\sum_Pe(\underline{x}, T_P)(\chern_n(G_\bullet)[T/P]_n)\notag
\end{eqnarray}
where the sum in the last two lines is taken over all minimal primes $P$ of $(x_1,\dots,x_d)$ of dimension $n$.

Finally, by \cite[Theorem 12.7.1]{RobertsMultiplicitiesandChernclassinLocalAlgebra} (or we can run the argument in the beginning of our proof), we have $\chern_n(G_\bullet)[T/P]_n=\chi_\infty(\overline{G}_\bullet)$ where $\overline{G}_\bullet$ denote the the complex $G_\bullet\otimes_TT/P$. At this point, we observe that $\overline{G}_\bullet$ is a complex of finite free $\overline{T}=T/P$-modules with finite length homology, and its length as a complex is exactly $n=\dim T/P$ (remember that $J$ is a perfect ideal of height $n$ and $G_\bullet$ is a minimal free resolution of $T/J$ over $T$). Hence by Theorem \ref{theorem--vanishing of higher Koszul}, we have $$\chern_n(G_\bullet)[T/P]_n=\chi_\infty(\overline{G}_\bullet)=\lim_{e\to\infty}\frac{l\left(H_0(F^e_{\overline{T}}(\overline{G}_\bullet))\right)}{p^{en}}
=\lim_{e\to\infty}\frac{l\left(\overline{T}/J^{[p^e]}\overline{T}\right)}{p^{e\cdot\dim \overline{T}}}=e_{HK}(J, T/P).$$
Now it is clear that (\ref{equation--interchanging two limits}) follows from (\ref{equation--intersection with divisors}).

\vspace{1em}

\noindent{\it\textbf{Second proof of (\ref{equation--interchanging two limits})}}: We will prove that
\begin{equation}
\label{equation--eHK(J+x) is the sum}
e_{HK}(J+\underline{x}, T)=\sum_Pe_{HK}(J, T/P)e(\underline{x}, T_P).
\end{equation}
This will establish (\ref{equation--interchanging two limits}) by (\ref{equation--equality on multiplicities and the crucial limit}) in our Lemma \ref{lemma--multiplicity for perfect ideal generically numerically Roberts}. To see (\ref{equation--eHK(J+x) is the sum}), note that by \cite[Corollary 4.11]{SmirnovEquimultiplicityinHilbertKunztheory} (applied to $I=(\underline{x})$ and $M=T$), we always have:
\begin{equation}
\label{equation--chains}
\sum_Pe_{HK}(J, T/P)e(\underline{x}, T_P)=\sum_Pe_{HK}(J, T/P)e_{HK}(\underline{x}, T_P)=\lim_{e\to\infty}\frac{1}{p^{en}}e_{HK}(J^{[p^e]}+\underline{x}, T).
\end{equation}
Next we recall that if $M$ is a finitely generated $T$-module of finite projective dimension, then $\Tor_i^T(M, T^{(e)})=0$ for all $i\geq 1$ (see \cite[Th\'{e}or\`{e}me (I.7)]{PeskineSzpiroDimensionProjective}). This implies that if $G_\bullet$ is a finite free resolution of $T/J$, then $F_T^{e}(G_\bullet)$ is a finite free resolution of $F_T^e(T/J)\cong T/J^{[p^e]}$. In particular, $J^{[p^e]}$ is a perfect ideal in $T$ for every $e$. Therefore we can apply (\ref{equation--equality on multiplicities and the crucial limit}) in Lemma \ref{lemma--multiplicity for perfect ideal generically numerically Roberts} to $J^{[p^e]}$ for every $e$ to obtain:
$$e_{HK}(J^{[p^e]}+\underline{x}, T)=\sum_Qe(\underline{x}, T/Q)e_{HK}(J^{[p^e]}, T_Q)$$ where the sum is taken over all minimal primes $Q$ of $J$ such that $\dim T/Q=d$. Since $\height J=n$, we have
$$\sum_Qe(\underline{x}, T/Q)e_{HK}(J^{[p^e]}, T_Q)=p^{en}\sum_Qe(\underline{x}, T/Q)e_{HK}(J, T/Q)=p^{en}e_{HK}(J+\underline{x}, T)$$ where the second equality is by (\ref{equation--equality on multiplicities and the crucial limit}) again. Therefore, the sequence $$\{\frac{1}{p^{en}}e_{HK}(J^{[p^e]}+\underline{x}, T)\}$$ is a constant sequence! Now (\ref{equation--eHK(J+x) is the sum}) follows immediately from (\ref{equation--chains}).
\end{proof}

\begin{remark}
In the second proof of Theorem \ref{theorem--limit in terms of local chern characters}, we crucially used \cite[Corollary 4.11]{SmirnovEquimultiplicityinHilbertKunztheory}, which in turn follows from some delicate ``uniform convergence" results in that paper (in the spirit of \cite{TuckerFsignatureExists}). Let us also point out that, once we combined (\ref{equation--equality on multiplicities and the crucial limit}) in Lemma \ref{lemma--multiplicity for perfect ideal generically numerically Roberts} with \cite[Theorem 5.17]{SmirnovEquimultiplicityinHilbertKunztheory}, we find that perfect ideals in complete (unmixed) local rings {\it satisfy colon capturing} of tight closure in the sense of \cite[Definition 5.15]{SmirnovEquimultiplicityinHilbertKunztheory}. We suspect that some version of this should be known in the literature, as tight closure of ideals of finite projective dimension (and more generally, of finite  phantom projective dimension) has been studied intensively \cite{HochsterHunekePhantomhomology}, \cite{Aberbachfinitephantomprojectivedimension}. However we have not been able to find a precise reference at the moment.
\end{remark}

We need one more lemma on the choice of system of parameters. We recall that a sequence of elements $y_1,\dots,y_n \in \m$ of $(R,\m)$ is called a {\it filter regular sequence} of $R$ if $y_i$ is not contained in any associated prime of $R/(y_1,\dots,y_{i-1})$ except $\m$ for every $i$. Standard prime avoidance shows that filter regular sequence (of any length) always exists.

\begin{lemma}
\label{lemma--choose generically CM sop}
Let $(T,\n_T)$ be a complete local ring with $T/\n_T$ an infinite field. Let $J\subseteq T$ be an ideal and $S=T/J$. Suppose $\height J=n\geq 1$ and $\dim S=d$. Then there exists $x_1,\dots,x_d\in \n_T$ such that
\begin{enumerate}
\item $x_1,\dots,x_d$ is part of a system of parameters of $T$;
\item $x_1,\dots,x_d$ is a minimal reduction of $\n_TS$ in $S$;
\item $T_P$ is Cohen-Macaulay for every minimal prime $P$ of $(x_1,\dots,x_d)$.
\end{enumerate}
\end{lemma}
\begin{proof}
We first notice that condition (3) is satisfied if $x_1,\dots,x_d$ is a filter regular sequence on $T/I$, where $I$ is the defining ideal of the non-Cohen-Macaulay locus of $T$. The reason is as follows. Since $\height I\geq 1$ and $\dim T\geq n+d\geq 1+d$, $I+(x_1,\dots,x_d)$ has height at least $\min\{\height I+d, \dim T\}\geq d+1$. Now suppose $P$ is a minimal prime of $(x_1,\dots,x_d)$. If $T_P$ is not Cohen-Macaulay, then $I\subseteq P$ and thus $I+(x_1,\dots,x_d)\subseteq P$. But then we have $$d\geq \height P\geq \height(I+(x_1,\dots,x_d))\geq d+1$$ which is a contradiction.

The remaining argument is standard. We can pick $x_1,\dots,x_d$ inductively and thus it suffices to construct $x_1$. First of all we want $x_1$ not contained in any minimal prime of $T$, and $x_1$ not contained in any associated prime of $T/I$ except possibly $\n_T$. There are only finitely many primes that we need to avoid. Call these $Q_1,\dots,Q_m$. Next we note that $\n_T/\n_T^2$ is a finite dimensional vector space over an infinite field $T/\n_T$. It is clear that $\{(Q_i+\n_T^2)/\n_T^2\}_{i=1}^{m}$, as well as the degree one elements of each minimal prime of $$\gr_\n S\cong \frac{T}{\n_T}\oplus \frac{\n_T}{\n_T^2+J}\oplus\frac{\n_T^2+J}{\n_T^3+J}\oplus\cdots $$ form a finite set of finite dimensional subspaces of $\n_T/\n_T^2$. Since none of these subspaces is equal to the whole $\n_T/\n_T^2$ and $T/\n_T$ is infinite, we can pick $x_1\in\n_T/\n_T^2$ that is not contained in all these subspaces. But this is precisely saying that $x_1$ is part of a system of parameter on $T$, is part of a minimal reduction of $S=T/J$, and is part of a filter regular sequence on $T/I$. Therefore we are done by the discussion above.
\end{proof}

We are now ready to prove our first inequality on multiplicities under flat local extensions between local rings.

\begin{theorem}
\label{theorem--inequalities on multiplicities in terms of difference of embdim}
Let $(R,\m)\to (S,\n)$ be a flat local map between complete local rings of characteristic $p>0$ with $\dim R=\dim S=d$. Suppose $R$ is a domain and $S/\n$ is algebraically closed. If $\edim S-\edim R=c$, then we have $$e(R)\leq \frac{c!}{2^c}e(S).$$
\end{theorem}
\begin{proof}
We first note that $c\geq 0$ by Theorem \ref{theorem--regularity defect of flat local extension}. Applying Lemma \ref{lemma--factoring maps with c.i. closed fibres}, we obtain $$(R,\m)\to (T,\n_T)\to (S,\n)=T/J$$ such that $T/\m T$ is a complete intersection, $J\subseteq \n_T^2$ is a perfect ideal of $T$, and $\pd_TS=c$. If $c=0$, then $S/\m S$ is a complete intersection by Corollary \ref{corollary--difference of embedding dimension <=1}. Hence $e(R)\leq e(S)$ follows from \cite{Lechinequalitiesofflatcouples} (see \cite[Theorem 1]{HerzogBLechHironakainequalities}). We thus assume $c\geq 1$, which implies $\height J=\pd_TS\geq 1$. Since $S/\n=T/\n_T$ is algebraically closed and $\height J\geq 1$, by Lemma \ref{lemma--choose generically CM sop} we can pick a minimal reduction $\underline{x}=x_1,\dots,x_d$ of $\n$ such that $x_1,\dots,x_d$ is part of a system of parameters of $T$ and $T_P$ is Cohen-Macaulay for every minimal prime of $(x_1,\dots,x_d)$. Now the hypotheses of Lemma \ref{lemma--multiplicity in terms of the limit} and Theorem \ref{theorem--limit in terms of local chern characters} are both satisfied. Applying them we obtain
\begin{equation}
\label{equation--interchanging two ideals e(S)=e_HK(J)}
e(S)=e(\underline{x}, S)=\lim_{e\to\infty}\frac{1}{p^{e\cdot \dim T}}\cdot e(\underline{x}, T^{(e)}\otimes_TS)=e_{HK}(J, T/(\underline{x})T).
\end{equation}
Because $J$ is perfect, we have $$\dim T/(\underline{x})T=\dim T-d=\pd_TS=c.$$
Since $J\subseteq \n_T^2$, we have
\begin{equation}
\label{equation--intermediate comparison of e_HK and e}
e_{HK}(J, T/(\underline{x})T)\geq e_{HK}(\n_T^2, T/(\underline{x})T)\geq \frac{1}{c!}e(\n_T^2, T/(\underline{x})T)=\frac{2^c}{c!}e(T/(\underline{x})T).
\end{equation}
Finally, it is well known that $e(T/(\underline{x})T)\geq e(T)$ when $\underline{x}$ is part of a system of parameters of $T$ \cite[Corollary 4]{SinghEffectofpermissibleblowup}. But $(R,\m)\to (T,\n_T)$ is a flat local map with $T/\m T$ a complete intersection, so $e(R)\leq e(T)$ follows from \cite{Lechinequalitiesofflatcouples} (see \cite[Theorem 1]{HerzogBLechHironakainequalities}). Therefore putting (\ref{equation--interchanging two ideals e(S)=e_HK(J)}) and (\ref{equation--intermediate comparison of e_HK and e}) together we get
$$e(S)\geq \frac{2^c}{c!}e(T/(\underline{x})T)\geq\frac{2^c}{c!}e(T)\geq \frac{2^c}{c!}e(R).$$
This finishes the proof of our theorem.
\end{proof}

\subsection{An inequality on multiplicities when $\edim S-\edim R$ is large} In this subsection we prove another inequality on the behavior of multiplicities under flat local extension. It gives a control on $e(S)$ when $\edim S-\edim R$ is large. We begin with some lemmas.
\begin{lemma}
\label{lemma--estimate on e_S involving higher Euler character}
Let $(S,\n)$ be a local ring of dimension $d$ and $N$ be a finitely generated $S$-module. Then for every system of parameters $\underline{x}=x_1,\dots, x_d$ of $S$, we have $$e(\underline{x}, N)\geq \nu_S(\n N)+(1-d)\nu_S(N)-\chi_1(\underline{x}, N).$$
\end{lemma}
\begin{proof}
By the Koszul characterization of multiplicity, we have
\begin{eqnarray*}
e(\underline{x}, N)&=&\sum_{i=0}^d(-1)^il_S(H_i(\underline{x}, N))\\
&=&l_S\left(\frac{N}{(\underline{x})N}\right)-\chi_1(\underline{x}, N)\\
&=&l_S\left(\frac{N}{\n(\underline{x})N}\right)-l_S\left(\frac{(\underline{x})N}{\n(\underline{x})N}\right)-\chi_1(\underline{x}, N)\\
&\geq& l_S(N/\n^2N)-d\cdot\nu_S(N)-\chi_1(\underline{x}, N)\\
&=& \nu_S(\n N)+(1-d)\nu_S(N)-\chi_1(\underline{x}, N)
\end{eqnarray*}
where the only $\geq$ is because $\n(\underline{x})\subseteq \n^2$ and we have a natural surjection $$\left(\frac{N}{\n N}\right)^{\oplus d}\twoheadrightarrow \frac{(\underline{x})N}{\n(\underline{x})N}.$$
This finishes the proof of the lemma.
\end{proof}

The next lemma is \cite[Proposition 4.2.3]{HanesThesis}. We give a proof for completeness. We note that it was assumed that $S$ is a flat local extension of $R$ in \cite{HanesThesis}. However, this condition is unnecessary in the proof.

\begin{lemma}
\label{lemma--Hanes' lemma}
Let $(R,\m)\rightarrow (S,\n)$ be a local map and let $M$ be a finitely generated module over $R$. Then:
$$\nu_S (\n M')\geq\nu_R(\m M)+(\edim S-\edim R)\cdot\nu_R(M)$$
where $M'=S\otimes_RM$.
\end{lemma}
\begin{proof}
The conclusion is obviously true if $M\cong R^n$: in this case we trivially have an equality. Therefore by induction, it suffices to prove the statement for $N=M/Ry$ where $y\in \m M$, assuming that it is true for $M$. We note that $N'=S\otimes_RN=M'/Sy$ and since $y\in \m M$, we have $\nu_R(M)=\nu_R(N)$.

If $y\in \m^2 M$, then the image of $y$ is in $\m^2 M'\subseteq \n^2 M'$, we have $\nu_R(\m M)=\nu_R(\m N)$ and $\nu_S(\n M')=\nu_S(\n N')$. So all the terms do not change when we pass from $M$ to $N$ and hence the conclusion holds for $N$. If $y\in \m M-\m^2M$, then $\nu_R(\m N)=\nu_R(\m M)-1$. Because the image of $y$ is in $\m M'\subseteq \n M'$, $\nu_S(\n N')\geq \nu_S(\n M')-1$. Since $\nu_R(N)=\nu_R(M)$, the remaining terms do not change, and we see that the inequality continues to hold for $N$.
\end{proof}

We are now ready to prove our second inequality on multiplicities under flat local extensions. Our strategy of the proof is inspired by \cite[Proposition 4.3.4]{HanesThesis}.\footnote{Hanes essentially proved this result under the additional (strong) hypothesis that $R$ admits a small maximal Cohen-Macaulay module and $R/\m$ is perfect (and the result was only stated in dimension three).} Our main new ingredients here are Lemma \ref{lemma--vanishing of higher Koszul}, Lemma \ref{lemma--multiplicity in terms of the limit}, and Lemma \ref{lemma--estimate on e_S involving higher Euler character}, which will drop the additional hypothesis in Hanes's argument.

\begin{theorem}
\label{theorem--inequalities on multiplicities when the difference of embdim is large}
Let $(R,\m)\to (S,\n)$ be a flat local map between complete local rings of characteristic $p>0$ with $\dim R=\dim S=d$. Suppose $R$ is a domain and $S/\n$ is algebraically closed. If $\edim S-\edim R=c\geq d$, then we have $$e(R)\leq \frac{d!}{2^d+c-d}e(S)\leq \frac{d!}{2^d}e(S).$$
\end{theorem}
\begin{proof}
The second inequality is trivial because we assumed $c\geq d$. Let $$(R,\m)\to (T,\n_T)\to (S,\n)=T/J$$ be a Cohen-factorization as in Theorem \ref{theorem--Cohen factorization}. We fix $\underline{x}=x_1,\dots,x_d$ a minimal reduction of $\n$. We apply Lemma \ref{lemma--estimate on e_S involving higher Euler character} to $\underline{x}$ and $N=T^{(e)}\otimes_TS$ (note that $T^{(e)}\otimes_TS$ is a finitely generated $S$-module because $T$ is complete with algebraically closed residue field) to get:
\begin{equation}
\label{equation--multiplicity after applying the first lemma}
e(\underline{x}, T^{(e)}\otimes S) \geq \nu_S(\n(T^{(e)}\otimes S))+(1-d)\cdot\nu_S(T^{(e)}\otimes S)-\chi_1(\underline{x}, T^{(e)}\otimes S).
\end{equation}
Next we apply Lemma \ref{lemma--Hanes' lemma} to $R=T$ and $M=T^{(e)}$ to get:
\begin{equation}
\label{equation--after applying Hanes' lemma}
 \nu_S(\n(T^{(e)}\otimes S))\geq \nu_T(\n_T T^{(e)})+(\edim S-\edim T)\cdot\nu_T(T^{(e)}).
\end{equation}
Combining (\ref{equation--multiplicity after applying the first lemma}) and (\ref{equation--after applying Hanes' lemma}) and noticing that $\nu_S(T^{(e)}\otimes S)=\nu_T(T^{(e)})$ because $S$ is a quotient of $T$, we have:
\begin{equation}
\label{equation--main estimates before dividing p^ed}
e(\underline{x}, T^{(e)}\otimes S)\geq \nu_T(\n_T T^{(e)})+(\edim S-\edim T+1-d)\cdot\nu_T(T^{(e)})-\chi_1(\underline{x}, T^{(e)}\otimes S).
\end{equation}
Now we observe that when $e\to\infty$, we have:
$$e(\underline{x}, T^{(e)}\otimes S) \rightarrow e(\underline{x},S)\cdot p^{e\cdot\dim T} \text{ by Lemma \ref{lemma--multiplicity in terms of the limit}}, $$
$$\nu(T^{(e)})=l_T(\frac{T^{(e)}}{\n_T T^{(e)}})\rightarrow e_{HK}(T)\cdot p^{e\cdot\dim T},$$
$$\nu(\n_T T^{(e)})=l_T(\frac{\n_T T^{(e)}}{\n_T^2T^{(e)}})=l_T(\frac{T^{(e)}}{\n_T^2T^{(e)}})-l_T(\frac{T^{(e)}}{\n_T T^{(e)}})\rightarrow (e_{HK}(\n_T^2,T)-e_{HK}(T))\cdot p^{e\cdot\dim T},$$
where the last two follow from the definition of the Hilbert-Kunz multiplicities and the fact that $T/\n_T$ is algebraically closed. More importantly, since we know that $J$ is a perfect ideal of $T$ by Lemma \ref{lemma--perfect ideal}, Lemma \ref{lemma--vanishing of higher Koszul} then implies
$$\chi_1(\underline{x}, T^{(e)}\otimes S)=\sum_{i=1}^d(-1)^{i-1}l_S(H_i(\underline{x}, T^{(e)}\otimes S))=o(p^{e\cdot \dim T}).$$
Hence after we divide (\ref{equation--main estimates before dividing p^ed}) by $p^{e\cdot\dim T}$ and let $e\to\infty$, we get:
\begin{eqnarray*}
\label{equation--main estimates after dividing p^ed}
e(\underline{x}, S)&\geq&(e_{HK}(\n_T^2,T)-e_{HK}(T))+(\edim S-\edim T+1-d)\cdot e_{HK}(T)\\
&=&e_{HK}(\n_T^2,T)+(\edim S-\edim T-d)\cdot e_{HK}(T)
\end{eqnarray*}
Finally, we apply Lemma \ref{lemma--Appendix} to obtain
\begin{eqnarray*}
e(\underline{x}, S)&\geq&e_{HK}(\m^2, R)+(\edim T-\edim R)\cdot e_{HK}(R)+ (\edim S-\edim T-d)\cdot e_{HK}(R)  \\
&=&e_{HK}(\m^2, R)+(\edim S-\edim R-d)\cdot e_{HK}(R)\\
&\geq& \frac{1}{d!}e(\m^2, R)+\frac{1}{d!}(c-d)\cdot e(R) \\
&=&\frac{2^d+c-d}{d!}\cdot e(R)
\end{eqnarray*}
where we use our assumption that $\edim S-\edim R=c\geq d$. Since $\underline{x}$ is a minimal reduction of $\n$, $e(S)=e(\underline{x}, S)$ and the above estimate immediately shows that $$e(R)\leq\frac{d!}{2^d+c-d}e(S).$$ This finishes the proof.
\end{proof}

\subsection{Proof of Theorem \ref{theorem--main technical theorem} in equal characteristic $p>0$} For the reader's convenience we restate our main theorem in characteristic $p>0$.
\begin{theorem}
\label{theorem--main theorem in characteristic p>0}
Let $(R,\m)\rightarrow (S,\n)$ be a flat local extension between local rings of equal characteristic $p>0$. If $\dim R=d$, then we have $$e(R)\leq \max\{1,\frac{d!}{2^d}\}\cdot e(S).$$ In particular, if $\dim R=3$, then $e(R)\leq e(S)$.
\end{theorem}
\begin{proof}
By Lemma \ref{lemma--reduction}, we may assume $\dim R=\dim S=d$, $R$ and $S$ are both complete, $R$ is a domain, and $S$ has algebraically closed residue field. Now the hypotheses of Theorem \ref{theorem--inequalities on multiplicities in terms of difference of embdim} and Theorem \ref{theorem--inequalities on multiplicities when the difference of embdim is large} are both satisfied. By Theorem \ref{theorem--regularity defect of flat local extension}, $c=\edim S-\edim S\geq 0$, thus Theorem \ref{theorem--inequalities on multiplicities in terms of difference of embdim} and Theorem \ref{theorem--inequalities on multiplicities when the difference of embdim is large} together tell us that $$e(R)\leq \max\{\frac{c!}{2^c}| 0\leq c\leq d\}\cdot e(S)=\max\{1, \frac{d!}{2^d}\}\cdot e(S).$$ This finishes the proof.
\end{proof}

\section{Reduction to characteristic $p>0$} In this section we will use reduction to characteristic $p>0$ to obtain Theorem \ref{theorem--main technical theorem} in characteristic $0$. The process of reducing Lech's conjecture in characteristic $0$ to characteristic $p>0$ is known at least to Mel Hochster, and most of the arguments are standard. Therefore we will omit the technical details in our presentation and direct the reader to the references (such as \cite{HochsterNonnegativityfollowingGabber} or \cite{DuttaAtheoremonSmoothness}) when necessary.

We should point out that, however, there are at least two nontrivial points when passing to characteristic $p>0$. First, we need to reduce Lech's conjecture to the case that $(R,\m)\to (S,\n)$ is a {\it module-finite} extension. We can only do this in characteristic $0$ in general. Second, we need to apply Artin approximation and reduction to characteristic $p>0$ while {\it preserving the multiplicities of $R$ and $S$}. This will be done by choosing minimal reductions of $\m$ and $\n$ and keeping track of the length of all the Koszul homology modules. We begin with the following lemma.

\begin{lemma}
\label{lemma--reduction to module finite}
Let $(R,\m)\to (S,\n)$ be a flat local map with $\dim R=d$. Suppose $R$ has equal characteristic $0$. In order to prove Lech's conjecture that $e(R)\leq e(S)$, or more generally, to prove $e(R)\leq C\cdot e(S)$ for certain constant $C$ depending only on $d$, we may assume $R\to S$ is a finite free extension.
\end{lemma}
\begin{proof}
By Lemma \ref{lemma--reduction}, we may assume $R$ and $S$ are both complete, $\dim R=\dim S=d$, and $S$ has algebraically closed residue field. Let $K$ be the coefficient field of $R$. Since we are in characteristic $0$, $K$ is contained in the coefficient field $L$ of $S$. Thus the two maps $R\to S$ and $L\to S$ agree on $K$. Hence we have the following natural maps: $$R\to R\widehat{\otimes}_KL\to S.$$

Let $R'=R\widehat{\otimes}_KL$ and $\m'=\m R'$. Note that $(R',\m')$ is still a complete local ring of dimension $d$, and we have $e(R)=e(R')$. Moreover, since $R'$ is flat over $R$ with $\m'=\m R'$, $\Tor^{R'}_1(R'/\m', S)=\Tor^R_1(R/\m, S)=0$. So by the local criterion of flatness, $R'\to S$ is still a flat local map. Now we can replace $R$ by $R'$ to assume $R\to S$ is a flat local map between complete local rings (of equal characteristic $0$) of the same dimension and same residue field. But it is well known that such an extension must be module-finite. Therefore we reduce to the case that $R\to S$ is a finite free extension.
\end{proof}

\begin{remark}
\label{remark--reduction to module finite not clear in general}
To the best of our knowledge, the above ``reduction to module-finite case" is not clear in characteristic $p>0$ (or in mixed characteristic) in general. In characteristic $p>0$, one might hope to use the same method as in Lemma \ref{lemma--reduction to module finite}. But the subtle point here is that the coefficient field of $R$ might not be contained in a coefficient field of $S$ and thus we cannot construct the desired $R'$ as in Lemma \ref{lemma--reduction to module finite}.
\end{remark}

We now start to prove Theorem \ref{theorem--main technical theorem} in characteristic $0$. There are two steps.

\subsection{Reduction to local rings essentially of finite type over a field} Suppose we have a counter-example $(R,\m)\to (S,\n)$ to Theorem \ref{theorem--main technical theorem}. By Lemma \ref{lemma--reduction to module finite}, we can assume that $(R,\m)\to (S,\n)$ is a finite free extension between complete local rings of dimension $d$, and $R$, $S$ have the same coefficient field $K=\overline{K}$ (of characteristic $0$). By Cohen's structure theorem we can fix $(A,\m_A)=K[[x_1,\dots,x_d]]\to R$ a module-finite extension. Now we think of this counter-example as a pair of finitely generated $A$-modules $R$, $S$ such that:
\begin{enumerate}[(1)]
\item $R$ and $S$ both have an algebra structure, they are both local of dimension $d$, with residue field $K$, and $S$ is finite free as an $R$-module;
\item $e(R)=\alpha$, $e(S)=\beta$, such that $\alpha>\max\{1,\frac{d!}{2^d}\}\beta$.
\end{enumerate}

The idea is to use Artin approximation \cite{ArtinApproximation} to replace this example by an example with the same properties but constructed over $B=(K[x_1,\dots,x_d]_{(x_1,\dots,x_d)})^h$, i.e., the Henselization of the ring $K[x_1,\dots,x_d]_{(x_1,\dots,x_d)}$. Since the Henselization is a direct limit of pointed \'{e}tale extensions of $K[x_1,\dots,x_d]_{(x_1,\dots,x_d)}$, it follows immediately that the original counter-example descends to $R\to S$, both essentially of finite type over $K$.

We think of $R$ generated over $A$ by $\theta_0=1, \theta_1,\dots, \theta_n$ and $S$ generated over $R$ by $\eta_0=1,\eta_1,\dots,\eta_r$, think of $S$ generated over $A$ by $\xi_{ij}$ (corresponding to $\theta_i\eta_j$). As an $A$-module, $R$ (resp., $S$) can be represented as the cokernel of a finite matrix $(a_{ij})$ (resp., $(b_{ij})$). To descend to the Henselization we want to think of $a_{ij}$ and $b_{ij}$ as solutions in $A$ of a finite system of polynomial equations over $B$. There will also be a lot of additional auxiliary elements involved in these equations, i.e., the system of polynomial equations will involve many variables besides those corresponds to $a_{ij}$ and $b_{ij}$. The idea is to construct a large family of equations satisfied by $a_{ij}$, $b_{ij}$, and those auxiliary variables such that when we take a new solution in $B$, congruent to the original solution modulo a certain high power of $\m_A$, we can use this solution to get a counter-example over $B$. Therefore, the key point here is to express (1) and (2) {\it equationally}.

%Note that the fact that $S$ is finite free as an $R$-module over $A$ using the generators $\theta_i\eta_j$ are then determined by $(a_{ij})$, because there are no relations among $\eta_0,\dots,\eta_r$ over $R$. We thus think of $S$ as the cokernel of a finite matrix $(b_{ij})$ (but we keep in mind that $(b_{ij})$ are determined from $(a_{ij})$).

The fact that $R$ and $S$ have an algebra structure can be expressed using equations is well known and can be found in many references (for example, see \cite{Smithtightclosureofparameterideals}). We can keep track of the dimensions of $R$ and $S$ using equations---this follows from \cite[Page 12, (9)]{HochsterNonnegativityfollowingGabber}. We can use equations to characterize certain element of $R$ winds up in $\m_A R$, this is done in \cite{Smithtightclosureofparameterideals}. Using this, we can then describe $R$ being a local ring with residue field $K$ equationally (and similarly for $S$) by first keeping track of the lengths of $R/\m_A R$ and $R/(\m_A+(\theta_1,\dots, \theta_n))R$---which follows from \cite[Page 12, (7)]{HochsterNonnegativityfollowingGabber} or \cite[Lemma 3.2]{DuttaAtheoremonSmoothness}, and then keeping track of the property that all generators of a fixed power of $(\theta_1,\dots, \theta_n)$ lies in $\m_A R$: this forces $(\theta_1,\dots, \theta_n)$ to be the only maximal ideal in the Artinian ring $R/\m_AR$ and thus $R$ is local with unique maximal ideal $\m_A+(\theta_1,\dots, \theta_n)$, whose residue field is $K$ (since we keep track that it has length $1$). The fact that $S$ is finite free over $R$ can be expressed using equations of $(a_{ij})$ and $(b_{ij})$: because we pick $\xi_{ij}=\theta_i\eta_j$ and so all $A$-relations of $\xi_{ij}$ are coming from $A$-relations of $\theta_i$. This shows that everything in (1) can be traced using equations.

Next we explain why (2) can be expressed equationally. This follows from the following more general results:
\begin{enumerate}[(i)]
\item We can keep track of a sequence of elements $\underline{y}=y_1,\dots,y_d\in R$ (resp., $\underline{z}=z_1,\dots,z_d\in S$) such that $(y_1,\dots,y_d)$ is a minimal reduction of $\m$ (resp., $(z_1,\dots,z_d)$ is a minimal reduction of $\n$) using equations.
\item We can keep track of the Euler characteristic $\chi(\underline{y}, R)$ (resp., $\chi(\underline{z}, S)$).
\end{enumerate}

To see (i), we set $y_i=r_{i1}\theta_1+\cdots+r_{in}\theta_n$ where $r_{ij}$ are (solutions) in $A$. To make sure that $(y_1,\dots,y_d)$ is a minimal reduction of $\m$, note that we can express $\m^N=(y_1,\dots,y_d)\m^{N-1}$ for some fixed $N$ using equations: this is clear since $\m^N=(y_1,\dots,y_d)\m^{N-1}$ amounts to say that for every $k_1+\cdots+k_n=N$, we have
\begin{equation}
\label{equation--describing a minimal reduction using equations}
\theta_1^{k_1}\cdots\theta_n^{k_n}=\sum_{l_1+\cdots+l_n=N-1} c_{il_1\cdots l_n}^{k_1\cdots k_n}y_i\theta_1^{l_1}\cdots\theta_n^{l_n} \hspace{1em} \text{ in } R.
\end{equation}
From (\ref{equation--describing a minimal reduction using equations}), we can further plug in $y_i=r_{i1}\theta_1+\cdots+r_{in}\theta_n$ and write each $c_{il_1\cdots l_n}^{k_1\cdots k_n}$ as $\sum s_j\theta_j$ where $s_j$ are (solutions) in $A$. We thus get many equations over $R$. But then these equations could be written as equations over $A$ using the variables introduced when we describe the multiplication structure on $R$ and extra equations involving $(a_{ij})$, the relation matrix representing $R$ as an $A$-module. We can do exactly the same thing for $z_1,\dots,z_d$ in $S$. Finally, (ii) follows from the fact that we can actually keep track of a finite complex of finitely generated modules as well as its homology \cite[Page 12, (8)]{HochsterNonnegativityfollowingGabber} (this was originated from \cite[Theorem 6.2]{PeskineSzpiroDimensionProjective}, see also \cite{DuttaAtheoremonSmoothness}). Note that we should be careful here because we also need to keep track that the complex is the Koszul complex of $y_1,\dots,y_d$ on $R$ (resp. the Koszul complex of $z_1,\dots,z_d$ on $S$), but we can add auxiliary equations to describe this property.

By the above discussion, we see that we can use a large family of equations to keep track of (1) and (2), i.e., a counter-example of Theorem \ref{theorem--main technical theorem}. Hence by Artin approximation \cite{ArtinApproximation}, these equations have a solution in $B=(K[x_1,\dots,x_n]_{(x_1,\dots,x_n)})^h$, i.e., we have a counter-example $(R,\m)\to (S,\n)$ where $R$ and $S$ are finite over $B$ with $R/\m=S/\n=K$.  Furthermore, since $B$ is a direct limit of local rings essentially of finite type over $K$. We know that we have a counter-example $(R,\m)\to (S,\n)$ of Theorem \ref{theorem--main technical theorem} where both $R$ and $S$ are local rings essentially of finite type over $K$ with $R/\m=S/\n=K$.

\subsection{Reduction to characteristic $p>0$} Since $R$, $S$ are local rings essentially finite type over $K=\overline{K}$ with $R/\m=S/\n=K$,\footnote{In fact we don't have to assume $K$ is algebraically closed and $R/\m=S/\n=K$ in the process of reduction to characteristic $p>0$. However, assuming these will simply the argument.} by Nullstellensatz we can do a change of variables if necessary to assume: $$(R,\m)\cong\left(\frac{K[y_1,\dots,y_m]}{I}\right)_{(y_1,\dots,y_m)} \text{ and } (S,\n)\cong\left(\frac{K[z_1,\dots,z_n]}{J}\right)_{(z_1,\dots,z_n)}.$$
Because $(R,\m)\to (S,\n)$ is a finite flat extension, we can find $f\in K[y_1,\dots,y_m]$ and $g\in K[z_1,\dots,z_n]$ such that $$\left(\frac{K[y_1,\dots,y_m]}{I}\right)_f\to\left(\frac{K[z_1,\dots,z_n]}{J}\right)_g$$ is a well-defined finite flat extension (note that the above two rings are both finite type over $K$). Now we can adjoin an extra variable to $y_1,\dots,y_m$ and $z_1,\dots,z_n$ respectively, and by Nullstellensatz we can perform a linear change of variables to assume that we have a finite flat extension: %\footnote{Here we are using $y_0, y_1,\dots,y_n$ and $z_0, z_1,\dots,z_m$ to denote the new set of variables, and $I$, $J$ to denote the new defining ideals, but these can definitely be different from the variables and ideals that we start with.}
$$R_K=\frac{K[y_0,y_1,\dots,y_m]}{I}\to \frac{K[z_0, z_1,\dots,z_n]}{J}=S_K,$$ such that $\m=(y_0, y_1,\dots,y_m)$, $\n=(z_0, z_1,\dots,z_n)$, and $R=(R_K)_\m\to S=(S_K)_\n$ is our counter-example to Theorem \ref{theorem--main technical theorem}. Furthermore we can assume that $\underline{y}=y_1,\dots,y_d$ is a minimal reduction of $\m$ and $\underline{z}=z_1,\dots,z_d$ is a minimal reduction of $\n$. We can assume that all the Koszul homology modules $H_i(\underline{y}, R_K)$ and $H_i(\underline{z}, S_K)$ are supported only at $\m$ and $\n$ respectively: they are supported at finitely many maximal ideals, so we can localize $R_K$ and $S_K$ at one extra element respectively (and then adjoin extra variables and perform a linear change of variables to assume $R_K$ and $S_K$ still have the shape as above) to assume they are only supported at $\m$ and $\n$ respectively.

At this point, we pick a finitely generated $\mathbb{Z}$-algebra $W$ of $K$ such that the coefficients of a set of generators of $I$, $J$ are contained in $W$. By generic freeness, we can shrink $W$ if necessary to assume we have a well-defined map between free $W$-algebras: $$R_W=\frac{W[y_0, y_1,\dots,y_m]}{I_W}\to \frac{W[z_0, z_1,\dots,z_n]}{J_W}=S_W.$$

By shrinking $W$, we may assume that $R_W\to S_W$ is still a finite flat extension and $\underline{y}$, $\underline{z}$ are still minimal reductions of $\m$, $\n$ respectively (here we still use $\m$ and $\n$ to denote the ideals ($y_0, y_1,\dots,y_m)$ and $(z_0, z_1,\dots,z_n)$ in $R_W$ and $S_W$ respectively). By generic freeness, we can shrink $W$ further to assume that all the kernels and cokernels of the Koszul complexes $K_\bullet(\underline{y}, R_W)$ and $K_\bullet(\underline{z}, S_W)$ are free $W$-modules, and the homologies are finite free $W$-modules. Therefore we have $$e((R_K)_\m)=\chi(\underline{y}, (R_K)_\m)=\chi(\underline{y}, R_K)=\sum(-1)^il_{R_K}(H_i(\underline{y}, R_K))=\sum (-1)^i\rank_WH_i(\underline{y}, R_W), $$ $$e((S_K)_\n)=\chi(\underline{z}, (S_K)_\n)=\chi(\underline{z}, S_K)=\sum(-1)^il_{S_K}(H_i(\underline{z}, S_K))=\sum (-1)^i\rank_WH_i(\underline{z}, S_W). $$

We pick a maximal ideal $Q$ of $W$ and note that $\kappa=W/Q$ is a field of characteristic $p>0$. Tensoring $R_W\to S_W$ with $\kappa$, we obtain a flat extension $R_\kappa\to S_\kappa$ (we can shrink $W$ to assume $S_W/R_W$ is free over $W$ and thus the map is an injection). We claim that $(R_\kappa)_\m\to(S_{\kappa})_\n$ is a counter-example to Theorem \ref{theorem--main technical theorem} in characteristic $p>0$, which would contradict our Theorem \ref{theorem--main theorem in characteristic p>0}. It is clear that $(R_\kappa)_\m\to(S_{\kappa})_\n$ is a flat local extension. Moreover, we know that $\underline{y}$ and $\underline{z}$ are minimal reductions of $\m$ and $\n$ in $(R_\kappa)_\m$ and $(S_\kappa)_\n$ respectively, because this is the case even in $R_W$ and $S_W$. Therefore, since the kernels, cokernels and homologies of the Koszul complexes $K_\bullet(\underline{y}, R_W)$ and $K_\bullet(\underline{z}, S_W)$ are all free $W$-modules, we have
$$H_i(\underline{y}, R_W)\otimes_W\kappa\cong H_i(\underline{y}, R_\kappa) \text{ and } H_i(\underline{z}, S_W)\otimes_W\kappa=H_i(\underline{z}, S_\kappa)$$
for every $i$. Because we can also assume that the homologies of $K_\bullet(\underline{y}, R_W)$ and $K_\bullet(\underline{z}, S_W)$ are annihilated by a power of $\m$ and $\n$ respectively by shrinking $W$ further, $K_\bullet(\underline{y}, R_\kappa)$ and $K_\bullet(\underline{z}, S_\kappa)$ are supported only at $\m$ and $\n$ respectively. Therefore we have
$$e((R_\kappa)_\m)=\chi(\underline{y}, (R_\kappa)_\m)=\chi(\underline{y}, R_\kappa)=\sum(-1)^il_{R_\kappa}(H_i(\underline{y}, R_\kappa))=\sum (-1)^i\rank_WH_i(\underline{y}, R_W),$$ $$e((S_\kappa)_\n)=\chi(\underline{z}, (S_\kappa)_\n)=\chi(\underline{z}, S_\kappa)=\sum(-1)^il_{S_\kappa}(H_i(\underline{z}, S_\kappa))=\sum (-1)^i\rank_WH_i(\underline{z}, S_W).$$
Thus $e((R_\kappa)_\m)=e((R_K)_\m)$ and $e((S_\kappa)_\n)=e((S_K)_\n)$.

\vspace{1em}

Theorem \ref{theorem--main technical theorem} is finally proved!
\bibliographystyle{skalpha}
\bibliography{CommonBib}
\end{document}